\documentclass[11pt]{article}
\usepackage[margin=.75in]{geometry}
\usepackage{multirow}
\usepackage{amsmath, amsthm, amssymb,fancyhdr,mathrsfs, slashbox, enumerate}
\usepackage{graphicx}
\usepackage{latexsym}
\usepackage[round,authoryear]{natbib}

\usepackage{graphics,amssymb,amsmath}
\usepackage{amsmath, amsthm, amssymb,fancyhdr,mathrsfs, slashbox, enumerate}
\usepackage{graphicx}
\usepackage{latexsym}
\usepackage{epsfig}
\usepackage{amsfonts,calc}
\usepackage{multicol}
\usepackage{graphics}
\usepackage{subfigure}
\usepackage{amsthm,amsmath}
\usepackage{graphicx}
\usepackage{float}

\newtheorem{theorem}{Theorem}
\newtheorem{lemma}{Lemma}
\newtheorem{remark}{Remark}

\newcommand{\mbold}[1]{\mbox{\boldmath $#1$}}
\newcommand{\X}{{\bf X}}
\renewcommand{\P}{\mathbb{P}}
\newcommand{\balpha}{\boldsymbol{\alpha}}

\newcommand{\blambda}{{\boldsymbol \lambda}}

\newcommand{\E}{\mathbb{E}}

\newcommand{\PROS}{\text{PROS}}
\newcommand{\SRS}{\text{SRS}}
\newcommand{\RSS}{\text{RSS}}

\begin{document}

\title{ Nonparametric Density Estimation   Using Partially Rank-Ordered Set Samples With Application in Estimating  the Distribution of Wheat Yield}

\author{
Sahar  Nazari$^a$, Mohammad Jafari Jozani$^{b,} $\footnote{Corresponding author: M$_{-}$Jafari$_{-}$Jozani@Umanitoba.CA. Phone: +1 204 272 1563.} , and Mahmood  Kharrati-Kopaei$^a$\\
\footnotesize{$^a$ Department of Statistics, Shiraz University, Iran.} \\
\footnotesize{$^b$ Department of Statistics,  University of Manitoba, 
                        Winnipeg, MB, Canada, R3T 2N2.}}
\maketitle 
\begin{abstract}
 
We study  nonparametric estimation of an unknown density function $f$  based on  the ranked-based observations obtained from a  partially rank-ordered
set (PROS) sampling design.    PROS sampling design has many applications in environmental, ecological and medical studies where  the exact measurement of the variable of interest is costly  but a small number of sampling units can be ordered with respect to the variable of interest  by  any means other than actual measurements   and this can be done at low cost. 
PROS   observations involve independent   order statistics which are not identically distributed and most of the commonly used nonparametric techniques are not directly applicable to   them.  
 We first  develop kernel density estimates of $f$  based on an imperfect PROS
sampling procedure and  study  its theoretical properties.  Then, we
 consider the problem  when the
underlying  distribution is  assumed to be symmetric and  introduce some  plug-in kernel density estimators of  $f$.    We use  an EM type algorithm to estimate  misplacement probabilities associated with an imperfect PROS design.  Finally, we expand on various numerical illustrations of our results  via  several simulation studies and a case study  to estimate the distribution of   wheat yield  using  the total acreage of land which is planted in wheat as an easily obtained  auxiliary information. 
 Our results show that the PROS
density estimate performs better than its SRS and RSS counterparts.  

\end{abstract}

\section{Introduction}

Nonparametric  density estimation techniques are widely  used to construct an estimate of a density function  and to provide valuable information about several features of the underlying population (e.g., skewness, multimodality, etc.) without imposing any parametric assumptions.   These methods are very popular in practice  and their advantages over histograms or parametric techniques are greatly appreciated. For example,   they are widely used  for both descriptive and analytic purposes  in economics to examine the distribution of income, wages and  poverty  (e.g., Minoiu and Reddy, 2012);  in   environmental and ecological studies   (e.g.,  Fieberg , 2007);   or in  medical research     (e.g.,  Miladinovic  et al. ), among others.

Most  of these density estimation techniques  are based on simple random sampling (SRS) design  which involve independent and identically distributed (i.i.d.) samples from the underlying population. There are only a few results available when the sampling design is different (e.g.,  Buskrik, 1998; Chen, 1999;  Gulati, 2004;  Lam et al., 2002;  Barabesi and Fattorini, 2002;  Breunig, 2001, 2008; Opsomer and Miller, 2005). The properties of nonparametric kernel  density estimation based on i.i.d.\ samples are well known and extensively studied in the literature (e.g.,  Wand and Jones, 1995; Silverman, 1986).  
In many applications, however,   the data sets  are often  generated  using more  complex  sampling designs and  they do not meet  the   i.i.d.\  assumption.   
 Examples include  the rank-based  sampling  techniques  which  are  typically   used when a small number  of  sampling units can be ordered fairly accurately with respect to a variable
of interest without actual measurements on them and this can also be done  at low cost.  This is a useful property since, quite often, exact measurements of these units can be very tedious and/or expensive. For example, for environmental risks such as radiation (soil contamination and disease clusters) or pollution (water contamination and root disease of crops), exact measurements  require substantial scientific processing of materials and a high cost as a result, while the variable of interest from a small number of experimental (sampling) units may easily be ranked. 
 These rank-based sampling designs   provide a collection of techniques to obtain more representative samples from the underlying population  with the help of the available auxiliary information. The samples obtained from such rank-based sampling designs often involve  independent  observations based on order statistics which are not identically distributed. So, it is important to develop kernel density estimators of the underlying population using such data sets  and  study their optimal properties.

 In this paper,  we study the problem of kernel density estimation  based on a  partially rank-ordered  set (PROS)  sampling design. 
  Ozturk (2011) introduced  PROS
sampling procedure as a generalization of the ranked set sampling (RSS) design.
 To obtain a ranked set sample  of
size $n$  one can proceed as follows. A set of $n$ units
is drawn from the underlying   population. The units are ranked via
some mechanism rather than the actual measurements of the variable
of interest. Then, only the unit ranked as the smallest is
selected for full measurement. Another set of $n$ units is drawn
and ranked and only the unit ranked as the second smallest is
selected for full measurement. This process is repeated $n$ times until the unit ranked the maximum is selected for the   final   measurement.
 See  Chen et al. (2003), Wolfe (2004, 2012) and references
therein for more details. 

In RSS,  rankers are forced to assign unique ranks to each observation even if they are not sure about the ranks. PROS design is  aimed at reducing the  impact of ranking error and the burden on  rankers by not requiring them to provide  a full ranking of all units in a set. Under PROS sampling technique,  rankers  have more flexibility by being able to  
divide the sampling units  into  subsets of pre-specified sizes based on their partial ranks. 
Ozturk (2011)   obtained  unbiased estimators for the
population mean and variance using PROS samples. He also showed that PROS sampling
has some advantages over RSS.  Hatefi and
Jafari Jozani (2013a, b) showed that the Fisher information of 
PROS samples is larger than the Fisher information of RSS samples of the same size.  Since 2011, 
 PROS sampling design has been the subject of many studies. Among others,  see  Gao and Ozturk (2012) for   a two-sample distribution-free inference;   Ozturk (2012) for  quantile estimation;   Frey (2012) for 
nonparametric estimation of the population mean;  Arslan and Ozturk (2013) for  parametric inference in a location-scale family of distributions, and   Hatefi et al. (2013) for finite 
 mixture model analysis based  on  PROS samples  with a fishery application.  
 
 The outline of this paper is as follows. In Section 2 we  introduce some preliminary results and present    a general theory
that  can be used to obtain   nonparametric estimates of some functionals of the  underlying distribution based on  imperfect PROS samples. In Section 3, we present
a nonparametric kernel  density estimate of the underlying distribution  based on an imperfect PROS sampling scheme
and  study its properties. We also consider the
problem of density estimation when the distribution of population
is symmetric.  In Section 4, we consider the problem of estimating the misplacement probabilities. To this end,   we propose a modified EM-algorithm to estimate the probabilities of subsetting errors. This algorithm is fairly simple to implement and simulation results show that its performance is satisfactory. In Section 5, we compare our PROS density estimate with its RSS and SRS counterparts using  simulation studies. Finally, in Section 6, we illustrate our proposed method by a real example.

\section{Necessary backgrounds and preliminary results}

In this section, we first give a formal introduction to PROS sampling design and then  present some   notations and preliminary results.  Also,  a general theory is obtained 
 to provide  nonparametric estimates of some functionals of the  underlying distribution based on  PROS samples. 

\subsection{PROS sampling design }

To obtain a PROS sample of size $N=nL$,  we   choose a set size $s$ and a design parameter  $D=\{d_1, \ldots, d_n\}$ that partitions the set  $\{1,\ldots, s\}$ into $n$
mutually exclusive subsets, where $d_j=\{(j-1)m+1,\ldots, jm\}$ and $m=s/n$. First $s$ units are randomly selected  from the underlying population 
and they  are assigned into subsets $d_j,  j=1,\ldots,n$,    without actual 
measurement of the variable of interest and only based on visual inspection or judgment,  etc. These subsets are partially judgment
ordered, i.e. all units in subset $d_j$ judged to have smaller ranks
than all units in $d_{j'}$, where $j<j'$. 
Then a unit is selected at random for measurement from the subset $d_1$ and it  is denoted 
by $X_{[d_1]1}$. Selecting another $s$ units assigning them 
into subsets, a unit is randomly drawn from subset $d_2$ and then  it is quantified and  denoted by $X_{[d_2]1}$. This process is repeated until we randomly draw a unit form   $d_n$ resulting in $X_{[d_n]1}$.
This constitutes  one  \textit{cycle} of PROS sampling technique. The cycle is then 
repeated $L$ times to generate a PROS sample of  size $N=nL$, i.e. $\X_{\PROS}=\{X_{[d_j]i}; j=1,\ldots,n; i=1,\ldots,L\}$.
 Table \ref{ta:pros-ex} shows the construction of a PROS sample with  $s=4,n=2, L=2$ and the  design parameter
$D=\{d_1,d_2\}=\{\{1,2\},\{3, 4\}\}$. Each set includes  
four units assigned into two partially ordered subsets such that  units in $d_1$ have  smaller ranks than units in $d_2$. In this subsetting process we do  not assign any ranks to units 
within each  subset  so that they  are equally likely to take any place in the
subset.  One unit, in each set from the bold faced subset, is randomly drawn and is quantified. The fully measured 
units are denoted by $X_{[d_j]i}$, $j=1, 2$; $ i=1, 2$. 
\begin{table}[h!]
\begin{center}
\caption{\small An example of PROS design}
\small{\begin{tabular}{cccc} \hline\hline
cycle & set &  Subsets & Observation \\ \hline
 1    & $S_1$ & $D_1=\{\mbold{d_1},d_2\}=\{ \mbold{\{1,2\}},\{3, 4\}\}$ & $X_{[d_1]1}$  \\ 
      & $S_2$ & $D_2=\{d_1,\mbold{d_2}\}=\{ \{1,2\},\mbold{\{3, 4\}} \}$ & $X_{[d_2]1}$  \\\hline
     2    & $S_1$ & $D_1=\{\mbold{d_1},d_2\}=\{ \mbold{\{1,2\}},\{3, 4\} \}$ & $X_{[d_1]2}$  \\
      & $S_2$ & $D_2=\{d_1,\mbold{d_2}\}=\{ \{1,2\},\mbold{\{3, 4\}} \}$ & $X_{[d_2]2}$  \\\hline\hline
 \end{tabular}}
  \label{ta:pros-ex}
 \end{center}  
 \end{table}

Note that if all units in $d_j$ have actually smaller ranks than
all units in $d_{j'}, j<j'$, then there is no subsetting error and
the PROS sample is perfect. Otherwise, we have subsetting error
and this PROS sample is called imperfect.  To model an imperfect PROS sampling design, following  Arslan and Ozturk (2013) and 
Hatefi and Jafari Jozani (2013a, b),   let ${\balpha}$ be a double stochastic  
misplacement probability 
matrix,
\begin{align}\label{alpha}
\balpha=\left[ \begin{array}{ccc}
\alpha_{d_1,d_1} & \cdots  & \alpha_{d_1,d_n} \\
\vdots           & \ddots  & \vdots           \\
\alpha_{d_n,d_1} & \cdots  & \alpha_{d_n,d_n} \end{array}
\right],
\end{align}
where $\alpha_{d_j,d_h}$ is the misplacement probability
of a unit from subset $d_h$ into subset $d_j$ with  $\sum^n_{h=1}\alpha _{d_j,d_h}=\sum^n_{j=1}\alpha_{d_j,d_h}=1$.  Throughout the paper,   we use $\PROS_{\balpha}(n, L, s, D)$ to denote an imperfect  PROS sampling  design  with subsetting error probability matrix $\balpha$, the number of subsets $n$, the number of cycles $L$,  the set size $s$,  and the design parameter $D=\{ d_j, j=1, \ldots, n \}$ where   $d_j=\{(j-1)m+1,\ldots, jm\}$,  in which  $m=s/n$ is the number of   unranked observations in each  subset.
 We note that SRS  and imperfect RSS (with ranking error probability matrix $\balpha$) can be expressed as special cases of the $\PROS_{\balpha} (n, L, s, D)$
 design  when $s=1$ and $s=n$, respectively. For a perfect PROS design,   since   $\alpha_{d_j,d_j}=1$
for $j=1,\dots,n$  and $\alpha_{d_j,d_h}=0$ for $h\ne j$,  we   use $\PROS_{\bf I}(n, L, s, D)$, where ${\bf I}$ is the identity matrix.

\subsection{Some notations and preliminary results}
In what follows, the probability
density function (pdf) and cumulative distribution function (cdf)
of the variable of interest  are denoted by $f$ and $F$, respectively. The
pdf and cdf of $X_{[d_j]i}$, for $i=1,\dots ,L$, are denoted by
$f_{[d_j]}$ and $F_{[d_j]}$, respectively,  and the pdf of the $r$-th order
statistic from a SRS of size $s$ is denoted by $f_{(r:s)}$.  We also use $i_k(g)$ to denote $\int x^kg(x)dx$ and  work with a second-order kernel density function  $K(\cdot)$ that is symmetric and satisfies the following conditions
$$i_0(K)= \int K(x)dx=1,  \quad i_0(K^2)= \int K^2(x)dx <\infty, \quad \text{and}\quad i_2(K)=\int x^2 K(x) dx <\infty. $$
The SRS, RSS and PROS density estimates  of $f$ are denoted by
$\hat{f}_{\SRS}$, $\hat{f}_{\RSS}$ and $\hat{f}_{\PROS}$,
respectively. Now, we present a useful lemma to show  the connection
between $f_{[d_j]}$ and $f$.

\begin{lemma}\label{Lemma 2.1} Let $\X_{\PROS}=\{X_{[d_j]i}; j=1,\dots,n;\
i=1,\dots,L\}$ denote a $\PROS_{\balpha}(n, L, s, D)$ sample of size $N$ from a population
with pdf $f$ and cdf $F$, respectively. Then 
\begin{align}\label{dj-u}
f_{[d_j]}(x)
&=nf(x)\sum^n_{h=1}\sum_{u \in d_h}\alpha_{d_j,d_h}{s-1 \choose u-1}F(x)^{u-1}{\overline{F}(x)}^{s-u}\nonumber\\
&= \frac{1}{m} \sum^n_{h=1}\sum_{u \in d_h}\alpha_{d_j,d_h} f_{(u:s)}(x),
\end{align}
where $\overline{F}(x)=1-F(x)$,  and
consequently
$$f(x)=\frac{1}{n}\sum^n_{j=1}f_{[d_j]}(x) \ \ \mbox{and} \ \ F(x)=\frac{1}{n}\sum^n_{j=1}F_{[d_j]}(x).$$
\end{lemma}

\begin{remark}\label{Remark 2.1} For a   $\PROS_{\bf I}(n, L, s, D)$ design, we have 
\begin{eqnarray*}
f_{[d_j]}(x)&=& nf(x)\sum_{u\in d_j}{s-1\choose u-1}F(x)^{u-1}\overline{F}{(x)}^{s-u} \nonumber\\
            &=& \frac{1}{m}\sum_{u\in d_j}f_{(u:s)}(x)=\frac{1}{m}\sum^{jm}_{r=(j-1)m+1}f_{(r:s)}(x).
\end{eqnarray*}
Therefore, a perfect PROS sample is a special case of an  imperfect
PROS sample and hence the results for a perfect PROS sampling  can be
obtained as a special case.
\end{remark}


Let $h(x)$ be a function of $x$  with $\mu_h=\E[h(X)]$. We  study   the method
of moments estimate of $\mu_h$ by using an imperfect PROS sampling
procedure, assuming that  the required moments of $h(X)$ exist.
Note that different choices  of $h(x)$ lead to different types
of estimators. For example,   $h(x)=x^l$ for $l=1,2,\dots$, corresponds to
the estimation of population moments;  $h(x)=\frac{1}{\lambda
}K\left(\frac{t-x}{\lambda }\right)$ where $K(\cdot )$ is a kernel 
function and $\lambda $ is a given constant, corresponds to the
kernel estimate of pdf and $h(x)=I(x\le c)$, where $I(A)$ is the
indicator function of  $A$, corresponds to the estimate of
cdf at point $c$.

The  method of moments estimate of $\mu_h$ based on an imperfect PROS sample
of size $N$ is given by
\begin{equation}\label{mme}
\hat{\mu}_{h.\PROS}=\frac{1}{N}\sum^L_{i=1}\sum^n_{j=1}h(X_{[d_j]i}).
\end{equation}
The properties of $\hat{\mu}_{h.\PROS}$ is discussed in
the following theorem.

\begin{theorem}\label{Theorem 2.1} Let $\X_{\PROS}= \{X_{[d_j]i}, j=1,\dots ,n;\
i=1,\dots ,L\}$ be a $\PROS_{\balpha}(n, L, s, D)$ sample of size $N$ from a population
with pdf $f$,  and let the method of moments estimator of
$\mu_{h}$ be defined as  in \eqref{mme}. Then
\begin{itemize}
\item[(i)] $\hat{\mu}_{h.PROS}$ is an unbiased estimator of $\mu_h$, i.e. $\E(\hat{\mu
}_{h.\PROS})={\mu}_h$.\

\item[(ii)] $var(\hat{\mu}_{h.PROS})\le
var(\hat{\mu}_{h.SRS})$ where $\hat{\mu}_{h.SRS}$ is the method of moments  estimator 
of $\mu_h$ based on a SRS sample of comparable size.

\item[(iii)] $\hat{\mu}_{h.PROS}$ is asymptotically
distributed as a normal distribution with mean $\mu_h$ and
variance $var(\hat{\mu }_{h.PROS})$ as ${L\to \infty}$.

\item[(iv)] $\hat{\mu}_{h.PROS}$ is a strong consistent
estimator of $\mu_{h.PROS}$ as $L \to \infty$.
\end{itemize}
\end{theorem}
\begin{proof}
The proof is  essentially the same as the one given by  Ozturk (2011) for $h(x)=x$ which we present here for the sake of completeness. Part $(i)$  is  an immediate consequence of Lemma \ref{Lemma 2.1}. For part
$(ii)$,  using Ozturk (2011),  we have
\begin{eqnarray*}
var(\hat{\mu}_{h.PROS})&=&var(\hat{\mu}_{h.SRS})-\frac{1}{n^2L}\sum^n_{j=1}(\mu_{h[d_j]}-\mu_{h})^2\\
                       &\le& var(\hat{\mu}_{h.SRS}),
\end{eqnarray*}
where ${\mu}_{h[d_j]}=\E[h(X_{[d_j]})]$. Note that the equality
holds if and only if $\mu_{h[d_j]}=\mu_h$ for  all $j=1,\dots,n$;
i.e. the subsetting process is purely random. For part $(iii)$,
note that
$\hat{\mu}_{h.\PROS}=\frac{1}{n}\sum^n_{j=1}\overline{h}[d_j]$,
where $\overline{h}[d_j]=\frac{1}{L}\sum^L_{i=1}h(X_{[d_j]i})$.
However, for a fixed $j$, $\overline{h}[d_j]$ converges
asymptotically to a normal distribution with mean
$\E[h(X_{[d_j]})]$ and variance $var[h(X_{{[d}_j]})]/L$ as $L\to
\infty$ by the Central Limit Theorem. Therefore, the result holds
for $\hat{\mu}_{h.PROS}$. Finally, part $(iv)$ follows from  the
Strong Law of Large Numbers.
\end{proof}

\section{Kernel density estimate  based on  PROS samples}

In this section, we present a kernel density estimator
of $f(x)$ based on an imperfect PROS sample of size $N$ and  study  some theoretical properties of our proposed estimator.
Also, we consider the problem of density estimation when $f(\cdot)$ is 
assumed to be symmetric. 

\subsection{Main results}

To obtain  a PROS kernel density estimator of $f(\cdot)$ we first  note that from
Lemma \ref{Lemma 2.1} we have $f(x)=\frac{1}{n}\sum^{n}_{j=1}f_{[d_j]}(x)$. For a fixed $j$, the sub-sample $X_{[d_j]i},
i=1,\dots,L$ can be considered as a simple random sample of size
$L$ from  $f_{[d_j]}(\cdot)$. Hence,
$f_{[d_j]}(x)$ can be estimated by the usual kernel method as follows
\begin{equation}\label{fhat-j}
\hat{f}_{[d_j]}(x)=\frac{1}{Lh}\sum^L_{i=1}{K\left(\frac{x-X_{[d_j]i}}{h}\right)},
\end{equation}
where $h$ is the bandwidth to be determined. We propose a
kernel estimate of $f(x)$  as
\begin{align}\label{kernel-estimate}
\hat{f}_{\PROS}(x)=\frac{1}{n}\sum^n_{j=1}\hat{f}_{[d_j]}(x)=\frac{1}{nLh}\sum^L_{i=1}\sum^n_{j=1}K\left(\frac{x-X_{[d_j]i}}{h}\right).
\end{align}
Now, we establish  some 
 theoretical properties of $\hat{f}_{\PROS}(x)$. To this end, let $\hat{f}_{\SRS}(x)=\frac{1}{N\, h}\sum^N_{j=1}K\left(\frac{x-X_j}{h}\right)$  be a kernel density estimator of $f(x)$ based on a SRS of size $N$ and note that  (e.g., Silverman, 1986)
 $$\E[\hat{f}_{\SRS}(x)]=f(x)+O(h^2),$$ 
 and  
 \begin{align}\label{var-srs}var(\hat{f}_{\SRS}(x))=\frac{1}{Nh}f(x)i_{0}(K^2)-\frac{1}{N}f^2(x)+O(\frac{h^2}{N}).
 \end{align}
  
\begin{theorem}\label{kernel-properties}  Suppose that $\hat{f}_{\PROS}(x)$ is a kernel density estimator of $f(x)$ based on a $\PROS_{\balpha}(n, L, s, D)$ sample of size
$N=nL$ and let $\hat{f}_{\SRS}(x)$ denote its corresponding SRS kernel estimator based on a SRS sample of the same size.  Then, 
\begin{itemize}
\item[(i)] $\E[\hat{f}_{\PROS}(x)]=\E[\hat{f}_{\SRS}(x)]$,

\item[(ii)] $var(\hat{f}_{\PROS}(x))=var(\hat{f}_{\SRS}(x))-\frac{1}{Nn}\sum^n_{j=1}(\mu
_{K[d_j]}-\mu_K)^2$,  where
$$\mu_{K[d_j]}=\E\left[\frac{1}{h}K\left(\frac{x-X_{[d_j]}}{h}\right)\right]\ \mbox{and} \
\mu_K=\E\left[\frac{1}{h}K\left(\frac{x-X}{h}\right)\right],$$
in which $X_{[d_j]}$ is an observation obtained from a  $\PROS_{\balpha}(n, L, s, D)$ design. 

\item[(iii)] $\hat{f}_{\PROS}(x)$ at a fixed point $x$ is
distributed asymptotically as a normal distribution with mean
$\E[\hat{f}_{PROS}(x)]$ and variance $var(\hat{f}_{\PROS}(x))$ for
large $L$.
\end{itemize}
\end{theorem}
\begin{proof}
The results hold immediately from Theorem  \ref{Theorem 2.1} by letting
$h(t)=\frac{1}{h}K(\frac{x-t}{h})$.
\end{proof}

Theorem \ref{kernel-properties} shows that $\hat{f}_{\PROS}(x)$ has the same
expectation as $\hat{f}_{\SRS}(x)$ and a smaller variance than
$\hat{f}_{\SRS}(x)$. This implies that $\hat{f}_{\PROS}(x)$ has a
smaller mean integrated  square error (MISE) than $\hat{f}_{\SRS}(x)$,
that is
$$\text{MISE}(\hat{f}_{\PROS})= \int{\E\left(\hat{f}_{\PROS}(x)-f(x)\right)^2dx}\le\int{\E\left(\hat{f}_{\SRS}(x)-f(x)\right)^2dx}=\text{MISE}(\hat{f}_{\SRS}).$$

\noindent In addition, by using part $(iii)$ of Theorem \ref{kernel-properties}, one
can construct an asymptotic pointwise $100(1-\nu)\%$ confidence interval for
$f(x)$ as follows
$$\hat{f}_{\PROS}(x)\pm z_{\nu/2}\sqrt{\widehat{var}_{\PROS}(x)},$$
where $z_{\nu/2}$ is the $100(1-\frac{\nu}{2})$-th quantile of the standard normal distribution and
$$\widehat{var}_{\PROS}(x)=\frac{1}{Nh}\hat{f}_{\PROS}(x)i_0(K^2)-\frac{1}{Nn} \sum^n_{j=1}\hat{f}^2_{[d_j]}(x),$$
 where $\hat{f}_{[d_j]}(x)$ is a consistent estimators of $f_{[d_j]}(x)$  given  by \eqref{fhat-j}.

Note that our estimate $\hat{f}_{\PROS}(x)$ depends on a bandwidth
$h$ which should be determined in practice. We present an
asymptotic optimal bandwidth by minimizing the asymptotic
expansion of MISE$(\hat{f}_{\PROS})$. In this regard, we first
present a lemma which is useful for obtaining the asymptotic
expansion of MISE$(\hat{f}_{\PROS})$.

\begin{lemma} \label{Lemma 3.1}Assuming that  the underlying density $f(\cdot)$ is sufficiently smooth with desired derivatives  and $K(\cdot)$ is  a second-order kernel function, for a fixed
$n$, as $h\to 0$, we have 
$$\E^2\left[\frac{1}{h}K\left(\frac{x-X}{h}\right)\right]-\frac{1}{n}\sum^n_{j=1}\E^2\left[\frac{1}{h}K\left(\frac{x-X_{[d_j]}}{h}\right)\right]=f^2(x)-\frac{1}{n}\sum^n_{j=1}f^2_{[d_j]}(x)+O(h^2).$$
\end{lemma}
\begin{proof}
Using \eqref{dj-u} and by changing the variable $v= \frac{x-t}{h}$,     we have
\begin{eqnarray*}
\E\left[\frac{1}{h}K\left(\frac{x-X_{[d_j]}}{h}\right)\right]&=&\int\frac{1}{h}K\left(\frac{x-t}{h}\right)f_{[d_j]}(t)dt\\
                                      &=&\frac{1}{m}\sum^n_{l=1}\sum_{u\in d_l}\alpha_{d_j,d_l}\int{K(v)f_{(u:s)}(hv+x)dv}.
\end{eqnarray*}
 Replacing 
$f_{(u:s)}(hv+x)$ by its Taylor expansion $f_{(u:s)}(x)+hvf'_{(u:s)}(x)+(hv)^2f^{''}_{(u:s)}(x)+O(h^2),$
and using the properties of the kernel function $K(\cdot)$, one can easily get
\begin{eqnarray*}
\E\left[\frac{1}{h}K\left(\frac{x-X_{[d_j]}}{h}\right)\right]&=&\frac{1}{m}\sum^n_{l=1}{\sum_{u \in d_l}{\alpha_{d_j,d_l}(f_{(u:s)}(x)+O(h^2))}}\\
                                                            &=&f_{[d_j]}(x)+O(h^2).
\end{eqnarray*}
Consequently,
$$\E^2\left[\frac{1}{h}K\left(\frac{x-X_{[d_j]}}{h}\right)\right]=f^2_{[d_j]}(x)+O(h^2),$$
and it is similarly verified that
$$\E^2\left[\frac{1}{h}K\left(\frac{x-X}{h}\right)\right]=f^2(x)+O(h^2),$$
which completes the proof.
\end{proof}

\begin{theorem}\label{Theorem 3.2} Suppose that the same bandwidth is
used in both $\hat{f}_{\SRS}$ and $\hat{f}_{\PROS}$. Then, for large
$N$,
$$\mbox{MISE}(\hat{f}_{\PROS})=\mbox{MISE}(\hat{f}_{\SRS})-\frac{1}{N}\Delta(f,n)+O(\frac{h^2}{N}),$$
where
$\Delta(f,n)=\int{[\frac{1}{n}\sum^n_{j=1}{f^2_{[d_j]}(x)-f^2(x)}]}dx$.
\end{theorem}
\begin{proof}
Note that
$bias(\hat{f}_{\PROS}(x))=\E[\hat{f}_{\PROS}(x)-f(x)]=bias(\hat{f}_{\SRS}(x))$.
Therefore,
\begin{eqnarray*}
\mbox{MISE}(\hat{f}_{\PROS})&=&\int{ \left[ var(\hat{f}_{PROS}(x))+bias^2(\hat{f}_{\PROS}(x))\right]dx}\\
                      &=&\int{ \left[var(\hat{f}_{\PROS}(x))+bias^2(\hat{f}_{\SRS}(x))\right]dx}.
\end{eqnarray*}
Now, from Lemma \ref{Lemma 3.1}
\begin{align}\label{var-pros}
var(\hat{f}_{\PROS}(x))=var(\hat{f}_{\SRS}(x))-\frac{1}{N}\left[\frac{1}{n}\sum^n_{j=1}f^2_{[d_j]}(x)-f^2(x)\right]+O(\frac{h^2}{N}).
\end{align}
Therefore, the result holds.
\end{proof}

Theorem \ref{Theorem 3.2} shows that the optimal bandwidth which minimizes
MISE$(\hat{f}_{\SRS})$, asymptotically minimizes
MISE$(\hat{f}_{\PROS})$ up to order $O(N^{-1})$. That is,  one can
use the following optimal bandwidth which is obtained by
minimizing asymptotical expansion of MISE$(\hat{f}_{\SRS})$
$$h_{opt.\SRS}=i_2(K)^{-2/5}\left[\frac{i_0(K^2)}{i_0(f^{''2})}\right]^{1/5}N^{-1/5},$$
(see Chen (1999)). 
%
The optimal bandwidth  $h_{opt.\SRS}$ depends on $f(\cdot)$ which is unknown and,   in practice, a nonparametric version of it can be used (see Silverman
(1986)). Theorem  \ref{Theorem 3.2} also shows that the PROS estimate reduces the
MISE of SRS estimate at order $O(N^{-1})$, and the amount of this
reduction asymptotically is $\frac{1}{N}\Delta(f,n)$ (note that
$\Delta(f,n)$ is non-negative). Unfortunately, the value of
$\Delta(f,n)$ depends on  $f(\cdot)$;  however, as we show below,  one can
characterize asymptotic   rate of this reduction in a perfect PROS sampling
procedure as an upper bound for $\Delta(f,n)$.

\begin{lemma} \label{Lemma 3.2} Under a $\PROS_{\bf I}(n, L, s, D)$  sampling design,
$$\frac{1}{n}\sum^n_{j=1}f^2_{[d_j]}(x)=nf^2(x)\,  \P\{Y=Z\},$$
where $Y$ and $Z$ are i.i.d.  binomial random variables with  parameters $s-1$ and
$F(x)$, i.e. $Y,Z\sim B(s-1,F(x))$.
\end{lemma}
\begin{proof}
Using  Remark \ref{Remark 2.1}, one can easily verify that
\begin{eqnarray*}
\frac{1}{n}\sum^{n}_{j=1}f^2_{[d_j]}(x)&=& nf^2(x)\sum_{j=1}^n\left[\sum^{jm}_{r=(j-1)m+1}{s-1\choose r-1}F(x)^{r-1}\overline{F}{(x)}^{s-r}\right]^2 \\
            &=& nf^2(x)\sum_{j=1}^{n}\P^2\{(j-1)m\le{Y}\le{jm-1}\},
\end{eqnarray*}
where $Y\sim B(s-1,F(x))$. Let $A_j=\{(j-1)m,\dots,jm-1\}$ for
$j=1,\dots,n$. Since $Z$ is  also distributed as a $B(s-1,F(x))$ distribution and it is 
independent of  $Y$, then
\begin{eqnarray*}
\sum_{j=1}^{n}\P^2\{(j-1)m\le{Y}\le{jm-1}\} &=& \sum_{j=1}^{n}\P\{Y=Z, Z \in A_j\} \nonumber\\
                                              &=& \P\{Y=Z, \bigcup_{j=1}^{n}(Z \in A_j)\} \nonumber\\
                                              &=& \P\{Y=Z\},
\end{eqnarray*}
since $A_j$s constitute  a disjoint partition of the set
$\{0,\dots,mn-1\}$ and this  completes the proof.
\end{proof}

By Lemma  \ref{Lemma 3.2}, we can derive an asymptotic result  which provides more insight into the rate of reduction in MISE in a perfect PROS sampling
procedure.

\begin{theorem} \label{Theorem 3.3} Under a   $\PROS_{\bf I}(n, L, s, D)$ sampling design,  we
have
$$\mbox{MISE}(\hat{f}_{\PROS})=\mbox{MISE}(\hat{f}_{\SRS})-\frac{1}{N}\left[\sqrt{\frac{n}{m}}\delta(f^2)-i_0(f^2)\right]-o(\frac{1}{Nm})+O(\frac{h^2}{N}),$$
where $\delta(f^2)=\int\frac{f^2(x)}{\sqrt{4\pi F(x)(1-F(x))}} dx$.
\end{theorem}
\begin{proof}
Note that by the Edgeworth expansion of $\P\{Y=Z\}$ in Lemma \ref{Lemma 3.2}, we get
$$\P\{Y=Z\}=\frac{1}{\sqrt{4 s \pi F(x)(1-F(x))}}+o(\frac{1}{s}).$$
Consequently, we can write
$$\frac{1}{N}\Delta(f,n)=\frac{1}{N}\left[\sqrt{\frac{n}{m}}\delta(f^2)-i_0(f^2)\right]-o(\frac{1}{Nm}),$$
and this completes the proof.
\end{proof}

Theorem \ref{Theorem 3.3} shows that a perfect PROS density estimate reduces
the MISE of $\hat{f}_{\SRS}$ at order $O(N^{-1})$ and this
reduction is increased by $\sqrt{n/m}$ linearly whenever $\sqrt{\frac{n}{m}}\delta(f^2)-i_0(f^2)$ is non-negative (a sufficient condition is $n\geq m
\pi$).  When $m=1$, the result is reduced to the
 result for perfect RSS density estimate given by Chen (1999). In Section 5, we compare  $\hat{f}_{\PROS}$ with   $\hat{f}_{\SRS}$ and  $\hat{f}_{\RSS}$ in a more general case where  the sampling procedure can be
either perfect or imperfect.

\subsection{Density estimation under symmetry assumption}
In this section, we consider the problem of  kernel density estimation  based
on an imperfect PROS sample of size $N=nL$ under the assumption
that $f(\cdot)$ is symmetric. To this end, suppose that $f(x)$ is symmetric about $\mu$,  that is
$f(x)=f(2\mu-x)$ for all $x$. One can easily  verify  that
$f_{[d_j]}(x)=f_{[d_{n-j+1}]}(2\mu-x)$ provided 
$\alpha_{d_j,d_h}=\alpha_{d_{n-j+1},d_{n-h+1}}$ for all
$j,h=1,\dots,n$. Therefore, based on the sub-sample $X_{[d_j]i},
i=1,\dots,L$, it is reasonable to estimate $f_{[d_j]}(x)$ by
$$\hat{f}^*_{[d_j]}(x,\mu)=\frac{1}{2}\left(\hat{f}_{[d_j]}(x)+\hat{f}_{[d_{n-j+1}]}(2\mu-x)\right),$$
where $\hat{f}_{[d_j]}(x)$ is given in \eqref{fhat-j}. Consequently, the
estimate of $f(x)$ under the symmetry assumption can be defined by
\begin{eqnarray*}
\hat{f}^*_{\PROS}(x,\mu)&=&\frac{1}{n}\sum_{j=1}^n\hat{f}^*_{[d_j]}(x,\mu)\\
                   &=&\frac{1}{2}\left(\hat{f}_{\PROS}(x)+\hat{f}_{\PROS}(2\mu-x)\right).
\end{eqnarray*}

Now, we consider the mean and the variance of $\hat{f}^*_{\PROS}(x,\mu)$ in the
following theorem.

\begin{theorem} \label{Theorem 3.4} Suppose that $f(x)$ is symmetric
about  $\mu$ and
$\alpha_{d_j,d_h}=\alpha_{d_{n-j+1},d_{n-h+1}}$ for all
$j,h=1,\dots,n$. Then,  based on an imperfect  $\PROS_{\balpha}(n, L, s, D)$  sample of size
$N=nL$, we have
\begin{itemize}
\item[(i)] $
\E\left[\hat{f}^*_{\PROS}(x,\mu)\right]=\E\left[\hat{f}_{\PROS}(x)\right]$,
\item[(ii)] $var\left(\hat{f}^*_{\PROS}(x,\mu)\right)\leq var\left(\hat{f}_{\PROS}(x)\right).$
\end{itemize}
\end{theorem}
\begin{proof}
Part $(i)$ is easily proved by the fact that
$\E\left[\hat{f}_{\PROS}(x)\right]=\E\left[\hat{f}_{\PROS}(2\mu-x)\right]$
under the symmetry assumption and Theorem  \ref{kernel-properties}. For  part $(ii)$,
note that for all $x$,  we have
$$ var(\hat{f}_{\PROS}(x))=var(\hat{f}_{\PROS}(2\mu-x)),$$ and consequently
\begin{eqnarray*}
var\left(\hat{f}^*_{\PROS}(x,\mu)\right)=\frac{1}{2}var\left(\hat{f}_{\PROS}(x)\right)+\frac{1}{2}cov\left(\hat{f}_{\PROS}(x),\hat{f}_{\PROS}(2\mu-x)\right).
\end{eqnarray*}
The result holds by using the Cauchy-Schwartz inequality.
\end{proof}

Theorem \ref{Theorem 3.4} shows that $\hat{f}^*_{\PROS}(x,\mu)$ has the same
bias as $\hat{f}_{\PROS}(x)$; however, it has smaller variance.
Therefore, under the symmetry assumption $\hat{f}^*_{\PROS}(x,\mu)$
has smaller MISE  and it dominates
$\hat{f}_{\PROS}(x)$. Note also  that if the symmetry point  $\mu$ is unknown, it can be estimated by
the PROS sample  to obtain a  plug-in estimator as  $\hat{f}^*_{\PROS}(x,\hat{\mu})$. Based on
a PROS sample of size $N$, several non-parametric estimators of
$\mu$ can be defined as follows
\begin{eqnarray}
\nonumber
\hat{\mu}_{1}&=&\frac{1}{N}\sum_{i=1}^L\sum_{j=1}^nX_{[d_j]i},\\ \nonumber
\hat{\mu}_{2}&=&median\bigl\{X_{[d_j]i},\quad i=1,\dots,L;j=1,\dots,n\bigr\}, \\ \nonumber 
\hat{\mu}_{3}&=&median\Biggl\{\frac{X_{[d_j]i}+X_{[d_k]l}}{2},\quad i,l=1,\dots,L ; j,k=1,\dots,n\Biggr\}, \\
\hat{\mu}_{4}&=&\frac{1}{L}\sum_{i=1}^Lmedian\bigl\{X_{[d_j]i},\quad j=1\dots,n\bigr\}.
\end{eqnarray}
Among these estimators $\hat{\mu}_{1}$ is the PROS sample mean which  is not robust against outliers, 
while $\hat{\mu}_{2}$, $\hat{\mu}_{3}$, and $\hat{\mu}_{4}$ are
robust estimators of $\mu$. Note that  $\hat{\mu}_{3}$ is a
Hodges-Lehmann type estimator of the location parameter. In Section 5, we consider the effect of these estimators on the MISE of $\hat{f}^*_{\PROS}(x,\hat{\mu})$.

\section{Estimating the misplacement probabilities}

So far we assumed that   the misplacement  probability  matrix $\balpha$ defined in \eqref{alpha} is given. In practice, the misplacement probabilities $\alpha_{d_j, d_h}$ are unknown and they should always be estimated.  This is a very important  problem as the performance of our kernel density estimator depends on the estimated values of $\alpha_{d_j, d_h}$. In this section, we use a modification of the EM algorithm of Arsalan and Ozturk (2013) to estimate $\alpha_{d_j,d_h}$'s.  We present the result for  a symmetric   misplacement probability matrix $\balpha$ with $\alpha_{d_j,d_h}=\alpha_{d_h,d_j}$. However, results for more general $\balpha$ can be obtained  by slight modifications of our results.  Let
$$\pi_{[d_j,d_h]i}=\frac{\alpha_{d_j,d_h}\bar{\beta}_{h}(F(X_{[d_j]i}))}{\sum_{h=1}^n\alpha_{d_j,d_h}\bar{\beta}_{h}(F(X_{[d_j]i}))},$$
in which
$$\bar{\beta}_{h}(F(X_{[d_j]i}))=\frac{1}{m}\sum_{u \in d_h}\beta_{u,s-u+1}(F(X_{[d_j]i})),$$
and $\beta_{a,b}(\cdot )$ denotes the pdf of a beta distribution with parameters $a$ and $b$.  Following Arslan and Ozturk (2013) we   estimate the misplacement probability matrix $\balpha$   through an iterative method. To this end, we start with an initial estimate of $\balpha$ say $\balpha^{(0)}$ which can be chosen to be a matrix associated with random subsetting with  $\alpha_{d_j, d_h}= \frac{1}{n}$.  Then,  we use the following iterative method:   
\begin{enumerate}
\item[(i)] For a given $\boldsymbol{\alpha}^{(t)}$ at step $t$ of the iterative process, calculate
$$w^{(t)}_{h',h}=\sum_{i=1}^L\pi^{(t)}_{[d_{h'},d_h]i}.$$

\item[(ii)] Calculate $$Q^{(t)}(\boldsymbol{\alpha})=\sum_{h=1}^n\sum_{h'=1}^nw^{(t)}_{h',h}\log(\alpha_{d_{h'},d_h}).$$

\item[ (iii)] Maximize $Q^{(t)}(\balpha)$ under the restrictions that the misplacement probabilities are symmetric and doubly stochastic and obtain the new $\boldsymbol{\alpha}$ and call it $\boldsymbol{\alpha}^{(t+1)}$. This can be done via a Lagrange multipliers method  to enforce the constraints as follows
\begin{align*}\mathcal{L}^{(t)}(\balpha, \blambda) 
&= \sum_{h=1}^n\left\{  \sum_{h'=1}^{h-1} w ^{(t)}_{h, h'} \log (\alpha_{d_{h'},d_h}) +  \sum_{h'=h}^{n} w ^{(t)}_{h, h'} \log (\alpha_{d_{h},d_{h'}})\right\} \\
&   +  \sum_{h=1}^n n\lambda_h\left\{  \sum_{h'=1}^{h-1}  \alpha_{d_{h'},d_h} +  \sum_{h'=h}^{n} \alpha_{d_{h},d_{h'}} -1\right\}, 
\end{align*}
where $\blambda=(\lambda_1, \ldots, \lambda_n)$. The details of this process are given in Ozturk (2010) as well as Arslan and Ozturk (2013). 
\item[(iv)] Repeat Steps (i)-(iii) till the sum of absolute error (SAE) of $\boldsymbol{\alpha}^{(t)}$ and $\boldsymbol{\alpha}^{(t+1)}$ is less than a predetermined value,  say $\delta$, that is 
$$\mbox{SAE}(\boldsymbol{\alpha}^{(t)},\boldsymbol{\alpha}^{(t+1)})=\sum_{i=1}^{n(n+1)/2}|\mathbf{\alpha}^{(t)}_i-\mathbf{\alpha}^{(t+1)}_i|\leq\delta.$$
\end{enumerate}
In practice, to calculate $\pi^{(t)}_{[d_j,d_h]i}$,  one can replace  $F(\cdot)$  by an estimate of $F$ such as the  empirical distribution function, i.e.
$$\hat{F}_{\PROS}(x)=\frac{1}{nL}\sum_{i=1}^L\sum_{j=1}^n I(X_{[d_j]i}\leq x).$$

\noindent To investigate the accuracy of our method, we perform  a small simulation  study when $n=m=3$, and $L=4,10$. Following Arslan and Ozturk (2013), we consider three misplacement probability matrices  $\balpha_1$, $\balpha_2$, and $\balpha_3$, where
$$\balpha_1=\left[ \begin{array}{ccc}
1 & 0 & 0 \\
0 & 1 & 0 \\
0 & 0 & 1 \end{array}
\right],
\balpha_2=\left[ \begin{array}{ccc}
0.900 & 0.075 & 0.025 \\
0.075 & 0.850 & 0.075 \\
0.025 & 0.075 & 0.900 \end{array}
\right],\quad\text{and} \quad 
\balpha_3=\left[ \begin{array}{ccc}
0.75 & 0.15 & 0.10 \\
0.15 & 0.70 & 0.15 \\
0.10 & 0.15 & 0.75 \end{array}
\right].
$$
\noindent We generate PROS  samples when the underlying population distributions are the standard Normal and Exponential distributions. For each distribution, the misplacement probabilities are estimated by using our proposed  iterative method with the help of  the  package ``Rsolnp" (Ghalanos and Theussl (2012) and Ye (1987)) in R  with $\delta=10^{-4}$. This process is repeated  100 times and the average of these estimates are used  as  the estimates of the misplacement probabilities. The values of the estimates and their corresponding  standard deviations (given in  parentheses)  are shown in Table \ref{table:Table1}. Note that following the  properties of $\balpha$ we  present the results for $\alpha_{d_{1}, d_{1}}$, $\alpha_{d_{1}, d_{2}}$, $\alpha_{d_{1},{d_3}}$, $\alpha_{d_{2},d_{2}}$, and $\alpha_{d_{2},d_{3}}$.
We observe  that the estimates are close to the true values and they have satisfactory biases given the fact  that  our proposed method is a fully nonparametric procedure and  the  sample size   is very small. We observe that  our proposed procedure slightly underestimates  $\alpha_{d_j, d_j}$,  especially  for $\balpha_1$. This is because  the perfect ranking model is at the boundary of the parameter space and as noted by Arslan and Ozturk (2013) the estimates are truncated whenever they exceed 1 due to the constraints on misplacement probabilities. However, the biases and standard deviations get smaller as the  cycle size increases.

\tabcolsep=0.11cm
\begin{table}[H]
\caption{{ The estimates and standard deviations (in  parentheses) of estimated misplacement probabilities when $n=m=3$, $L=4,10$  and the underlying distributions are the standard Normal and Exponential distributions.}}\label{table:Table1}
\bigskip
\centering
\small
\begin{tabular}{l c c c c c c c}
\hline \hline
  Distributions & $\balpha$ & $L$ & $\hat{\alpha}_{d_{1},d_{1}}$ & $\hat{\alpha}_{d_{1}, d_{2}}$ & $\hat{\alpha}_{d_{1}, d_{3}}$ & $\hat{\alpha}_{d_{2}, d_{2}}$ & $\hat{\alpha}_{d_{2}, d_{3}}$ \\
\hline
              & $\balpha_1$  & 4 & 0.9640(0.099) & 0.0355(0.097) & 0.0006(0.005) & 0.9179(0.155) & 0.0466(0.109) \\[-1ex]
              &                      & 10& 0.9775(0.051) & 0.0204(0.051) & 0.0021(0.012) & 0.9565(0.070) & 0.0232(0.054) \\[0.1cm]
\cline{2-8}
              & $\balpha_2$  & 4 & 0.8934(0.156) & 0.0717(0.140) & 0.0348(0.087) & 0.8247(0.198) & 0.1036(0.143) \\[-1ex]
Normal        &                      & 10& 0.8961(0.106) & 0.0759(0.096) & 0.0280(0.044) & 0.8356(0.153) & 0.0886(0.102) \\[0.1cm]
\cline{2-8}
              & $\balpha_3$  & 4 & 0.7578(0.207) & 0.1535(0.188) & 0.0888(0.116) & 0.6583(0.274) & 0.1882(0.217) \\[-1ex]
              &                      & 10& 0.7365(0.140) & 0.1605(0.137) & 0.1030(0.082) & 0.6902(0.174) & 0.1493(0.125) \\[0.1cm]
\cline{1-8}
              & $\balpha_1$  & 4 & 0.9384(0.134) & 0.0612(0.134) & 0.0004(0.004) & 0.8751(0.185) & 0.0637(0.139) \\[-1ex]
              &                      & 10& 0.9730(0.052) & 0.0265(0.052) & 0.0005(0.005) & 0.9501(0.068) & 0.0234(0.048) \\[0.1cm]
\cline{2-8}
              & $\balpha_2$  & 4 & 0.8918(0.167) & 0.0833(0.164) & 0.0249(0.059) & 0.7880(0.225) & 0.1287(0.158) \\[-1ex]
Exponential   &                      & 10& 0.8802(0.107) & 0.0891(0.104) & 0.0307(0.042) & 0.8014(0.161) & 0.1095(0.112) \\[0.1cm]
\cline{2-8}
              & $\balpha_3$  & 4 & 0.7486(0.232) & 0.1475(0.187) & 0.1039(0.136) & 0.6535(0.254) & 0.1990(0.204) \\[-1ex]
              &                      & 10& 0.7515(0.132) & 0.1513(0.125) & 0.0972(0.090) & 0.6893(0.173) & 0.1595(0.128)\\[0.1cm]
\hline
\end{tabular}
\end{table}

\section{ Simulation Study}

In this section, we compare  the performance of $\hat{f}_{\PROS}$ with its SRS
and RSS counterparts. We first discuss the asymptotic reduction
rate in  the variance (RRV) of $\hat{f}_{\SRS}$ and $\hat{f}_{\RSS}$
by using   a  PROS density estimate $\hat{f}_{\PROS}$. We then compare the MISE$(\hat{f}_{\PROS})$ with MISE$(\hat{f}_{\SRS})$ and MISE$(\hat{f}_{\RSS})$.
Finally,  we consider the effect of estimating the symmetry point on the MISE of $\hat{f}^*_{\PROS}(x,\hat{\mu})$ when we assume that the underlying distribution is symmetric.
\subsection{Comparing the reduction in variances}
Using  \eqref{var-pros},  the RRV  of $\hat{f}_{\PROS}$ over 
  $\hat{f}_{\SRS}$ that measures at what rate $\hat{f}_{\PROS}$
reduces the asymptotic variance of $\hat{f}_{\SRS}$ can be defined
as
\begin{eqnarray*}
\mbox{RRV}(\hat{f}_{\PROS},\hat{f}_{\SRS})&=&\frac{\frac{1}{n}\sum_{j=1}^n{f^2_{[d_j]}(x)}-f^2(x)}{\frac{1}{n}\sum_{j=1}^n{f^2_{[d_j]}(x)}}\\
   &=&1-{\left[n\sum_{j=1}^{n}{\left\{\sum_{r=1}^{n}{\sum_{u \in d_r}{\alpha_{d_j,d_r}{s-1 \choose u-1}p^{u-1}(1-p)^{s-u}}}\right\}^2}\right]}^{-1},
\end{eqnarray*}
where $p=F(x)$. It is clear that $\mbox{RRV}(\hat{f}_{\PROS},\hat{f}_{\SRS})$ is a nonparametric measure which  does not depend on
the  underlying distribution  function. Note  that
if RRV$(\hat{f}_{\PROS},\hat{f}_{\SRS})=0$ at certain percentiles
$p$, then $\hat{f}_{\PROS}$ and $\hat{f}_{\SRS}$ have equal
 variances at these percentiles asymptotically. However, if
RRV$(\hat{f}_{\PROS},\hat{f}_{\SRS})=\beta>0$, then $\hat{f}_{\PROS}$  reduces the variance of  $\hat{f}_{\SRS}$ at order $O(N^{-1})$
and this reduction increases linearly at rate $\beta$. 
%
%
The values of RRV$(\hat{f}_{\PROS},\hat{f}_{\SRS})$ can be easily
calculated when $n$, $m$ and the misplacement probabilities
$\alpha_{d_i,d_j}$ are given. For $m=3$ and $n=2,\dots,7$ and
misplacement probabilities $\alpha_{d_i,d_i}=\alpha_0$ and
$\alpha_{d_i,d_j}=(1-\alpha_0)/(n-1)$ for $i\neq j$, the values of
RRV$(\hat{f}_{\PROS},\hat{f}_{\SRS})$ are presented in Figure \ref{fig:Figure1} when
$\alpha_0=0, 0.3, 0.7, 1$.

We observe that  for all values of $\alpha_0$ the amount   of RRV
increases symmetrically as $p$ gets away  from 0.5 to 0 and 1. This
shows that the best performance of  the  PROS density estimate over its SRS
counterpart  happens at the tail of the distribution. When $n=2$, the
PROS and SRS estimates have equal precision at $p=0.5$; otherwise,
the PROS estimate reduces the variance of SRS estimate. When
$\alpha_0=0$, the value of RRV decreases when $n$ increases. This
suggests using a small sample size when the misplacement
probabilities (ranking errors) are large. We also note that RRV increases as
both $\alpha_0$ and $n$ increase.   The best performance of PROS design over SRS design happens when the subsetting is either perfect or it is moderately good, that is when   $\alpha_0=1$ or $\alpha_0=0.7$, respectively. 
Similar results  are  observed when  $m=5$ which we do not present here.

 To obtain  the RRV of $\hat{f}_{\PROS}$ over  $\hat{f}_{\RSS}$
 we first note that (see  Chen (1999))
$$var(\hat{f}_{\RSS}(x))=var(\hat{f}_{\SRS}(x))-\frac{1}{N}\left[\frac{1}{n}\sum_{r=1}^n{f^2_{[r]}(x)}-f^2(x)\right]+O(\frac{h^2}{N}).$$
 Now,  using \eqref{var-pros}, the
RRV of $\hat{f}_{\PROS}$ when $\hat{f}_{RSS}$ is 
defined as
$$\mbox{RRV}(\hat{f}_{\PROS},\hat{f}_{RSS})=\frac{\frac{1}{n}\sum_{j=1}^n{f^2_{[d_j]}(x)}-\frac{1}{n}\sum_{r=1}^n{f^2_{[r]}(x)}}{\frac{1}{n}\sum_{j=1}^n{f^2_{[d_j]}(x)}},$$
where
\begin{eqnarray*}
\frac{1}{n}\sum_{r=1}^n{f^2_{[r]}(x)}=nf^2(x)\sum_{r=1}^{n}{\left[\sum_{k=1}^{n}{p_{rk}{n-1
\choose k-1}p^{k-1}(1-p)^{n-k}}\right]^2},
\end{eqnarray*}
in which $p=F(x)$ and $p_{rk}$ for $r,k=1,\dots,n$ are
the ranking error  probabilities in an imperfect RSS procedure.

\begin{figure}[H] \label{fig:Figure1} 
  \begin{minipage}[b]{0.45\linewidth}\centering
    \includegraphics[width=1\linewidth]{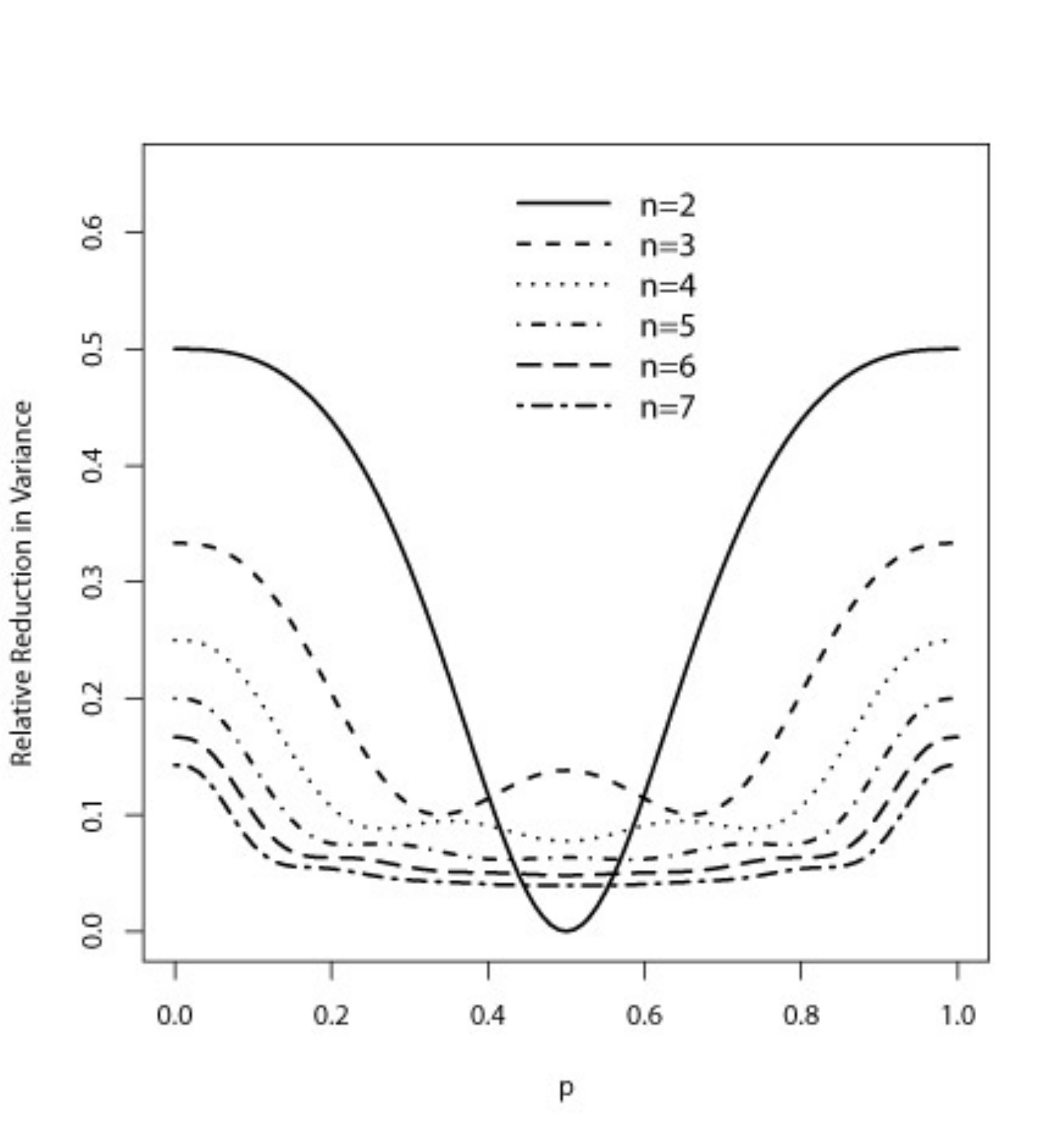} \\
   (a)~ $\alpha_0=0$
  \end{minipage} 
  \begin{minipage}[b]{0.45\linewidth}\centering
    \includegraphics[width=1\linewidth]{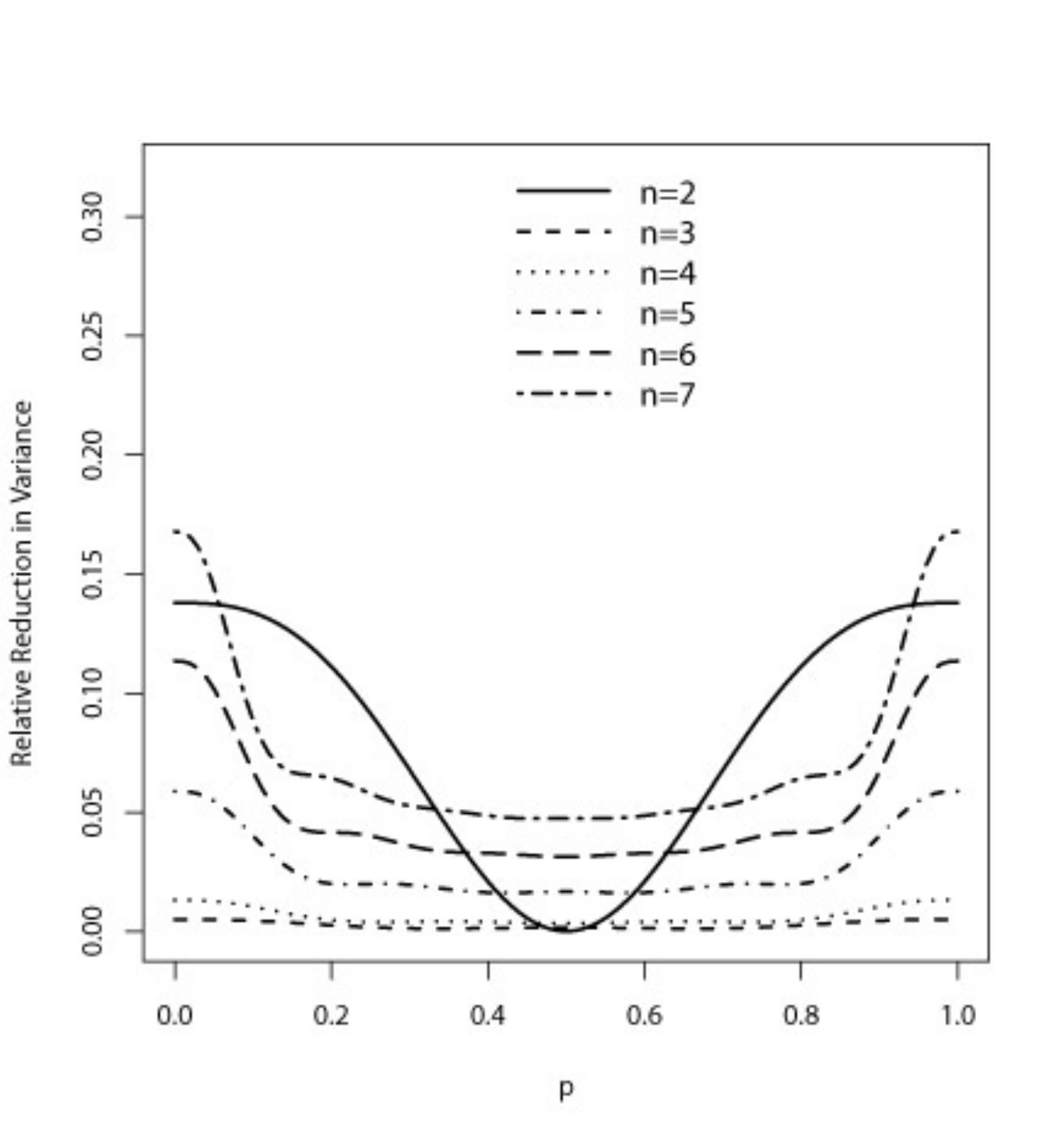} \\
  $\alpha_0= 0.3$
  \end{minipage} 
  \begin{minipage}[b]{0.45\linewidth}\centering
    \includegraphics[width=1\linewidth]{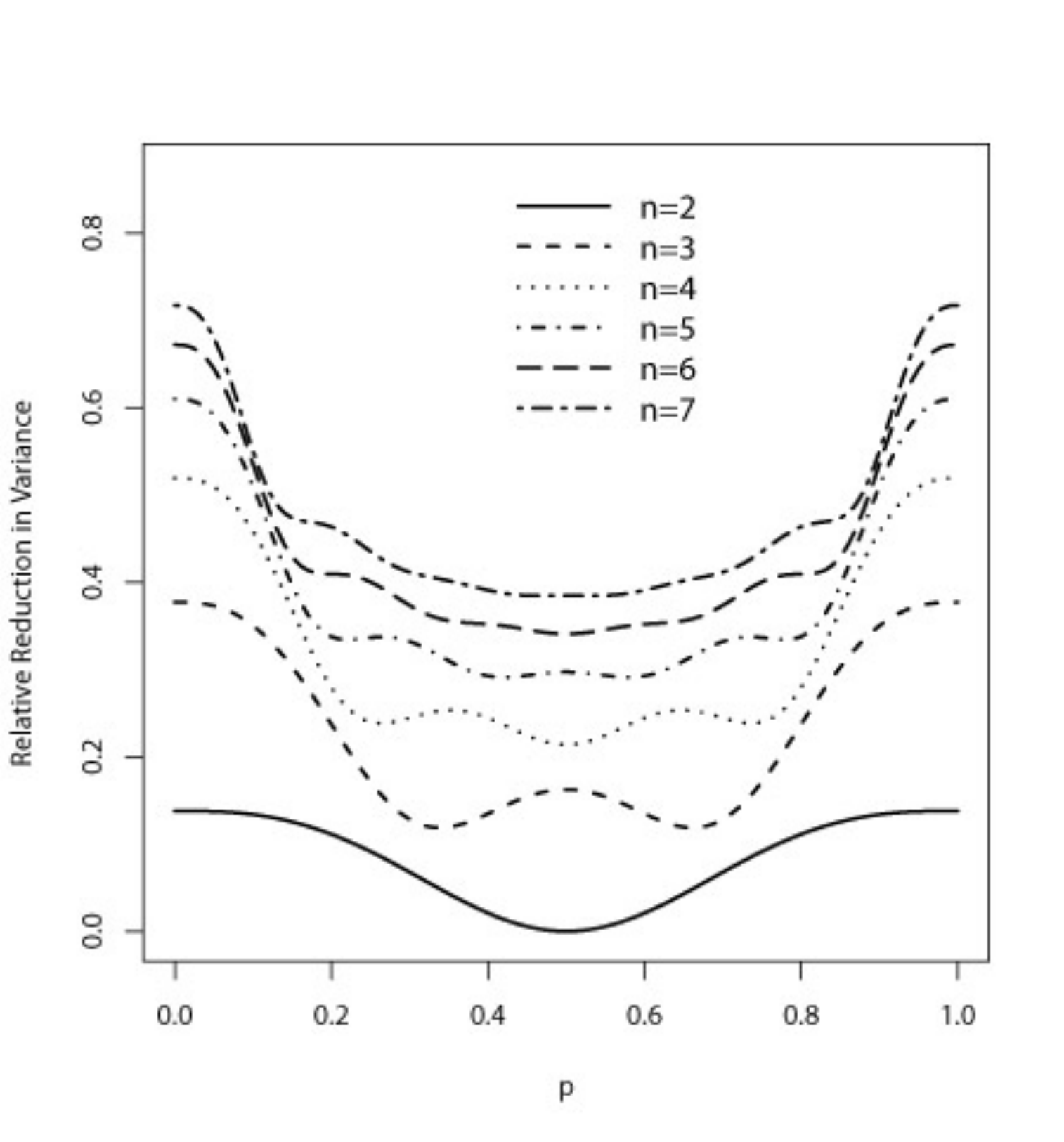}  \\
    (c)~$\alpha_0= 0.7$
  \end{minipage}
  \hfill
  \begin{minipage}[b]{0.45\linewidth}\centering
    \includegraphics[width=1\linewidth]{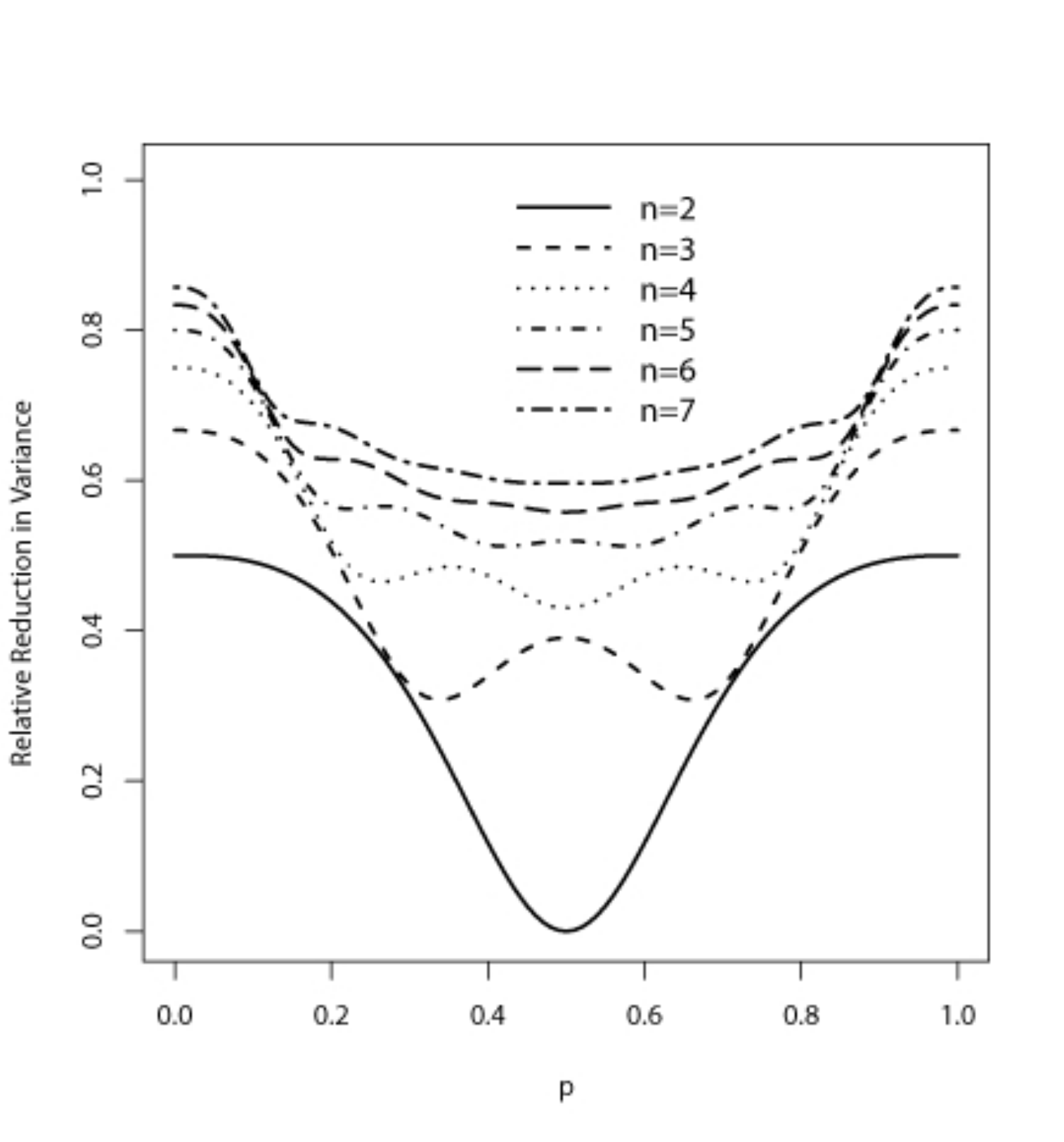}  \\
    (d)~$\alpha_0= 0.3$
  \end{minipage} 
  \caption{\large{RRV$(\hat{f}_{PROS},\hat{f}_{SRS})$ for $n=2,\dots,7$
when $m=3$ and $\alpha_0=0,0.3,0.7,1$.}}
\end{figure}
%
For $m=3$, $n=2,\dots,7$ and the ranking error 
probabilities equal to the misplacement error probabilities in  its corresponding   imperfect PROS design, the values of RRV$(\hat{f}_{\PROS},\hat{f}_{\RSS})$ are
presented in Figure \ref{fig:Figure2}. It is seen that the values of RRV are  symmetric about $p=0.5$. When $\alpha_0=0$, the RRV
decreases as $n$ increases, and  by increasing  $\alpha_0$  RRV of $\hat{f}_{\RSS}$ increases as $n$ increases. The RRV of RSS
is zero when $p=0$ and 1 (when $n=2$, the value of RRV is also
zero at $p=0.5$). This means that the PROS and RSS estimates have
the same precision at these percentiles. The maximum value of RRV is more than 35 percent when the sampling procedure
is perfect and $n=7$. Similar results are obtained when $m=5$ which are not presented  here. 

\begin{figure}[H] \label{fig:Figure2} 
  \begin{minipage}[b]{0.45\linewidth}\centering
    \includegraphics[width=1\linewidth]{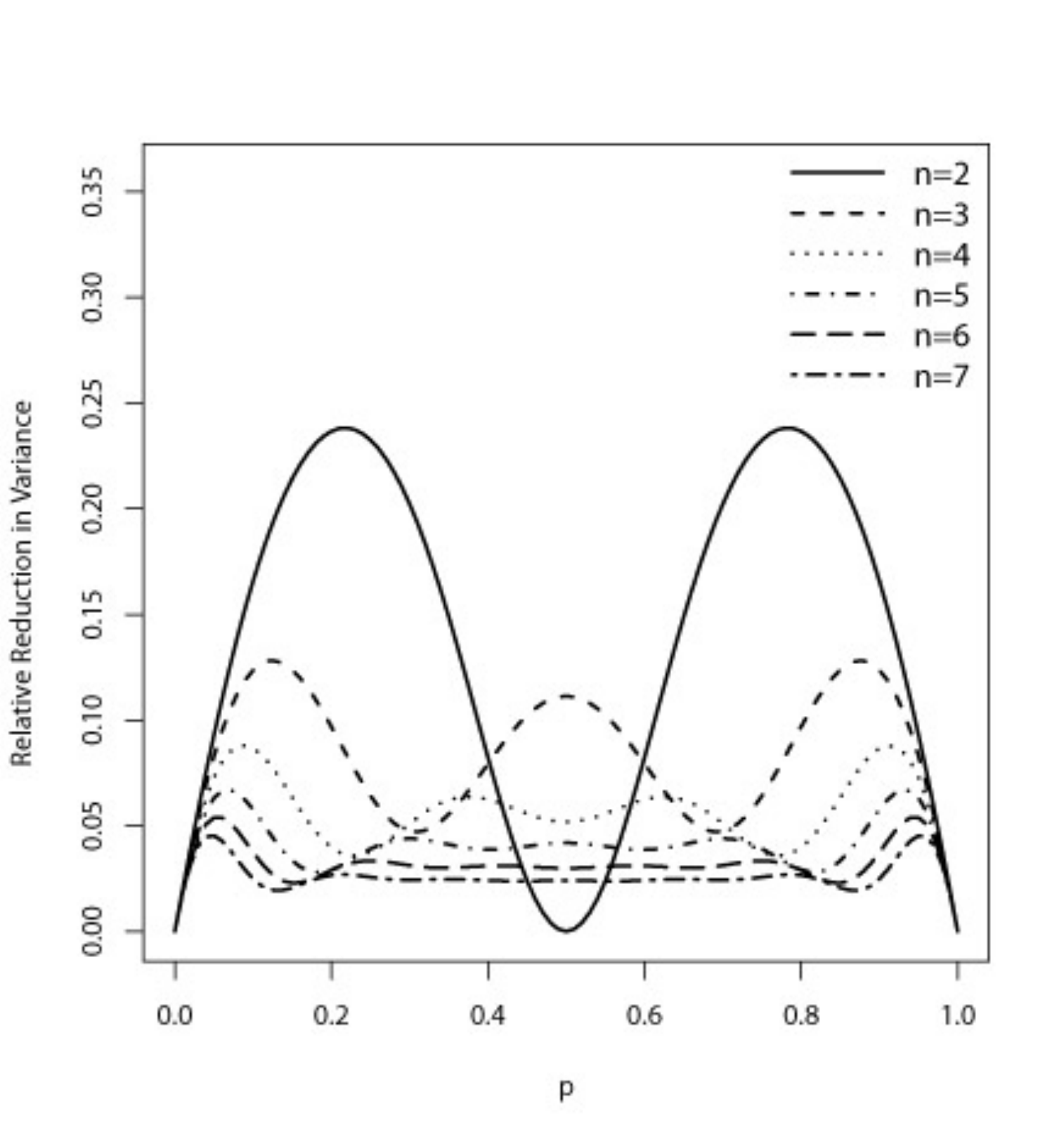} \\
  (a)   $\alpha_0=0$ 
  \end{minipage} 
  \begin{minipage}[b]{0.45\linewidth}\centering
    \includegraphics[width=1\linewidth]{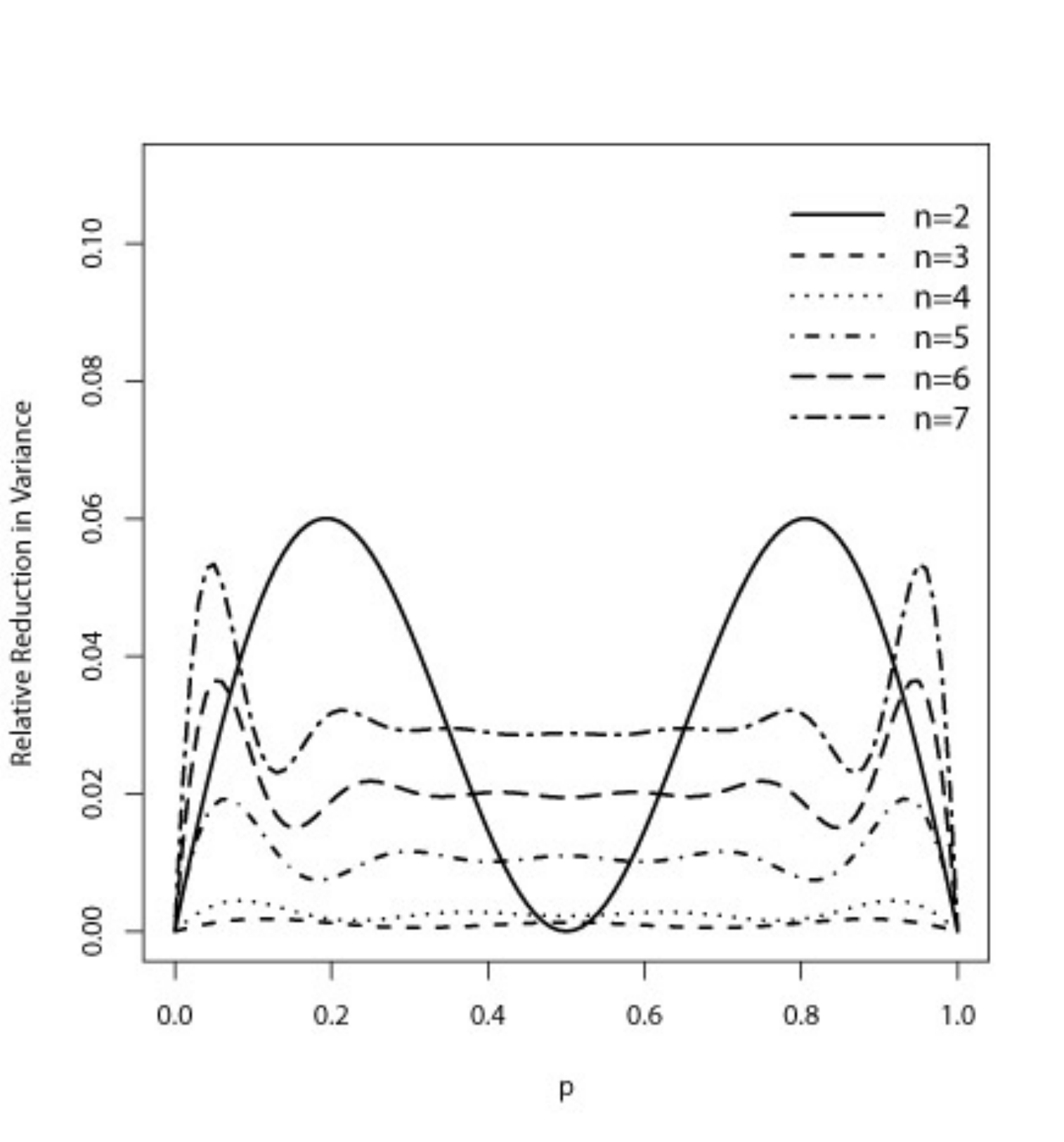} \\
  (b) $\alpha_0=0.3$ 
  \end{minipage} 
  \begin{minipage}[b]{0.45\linewidth}\centering
    \includegraphics[width=1\linewidth]{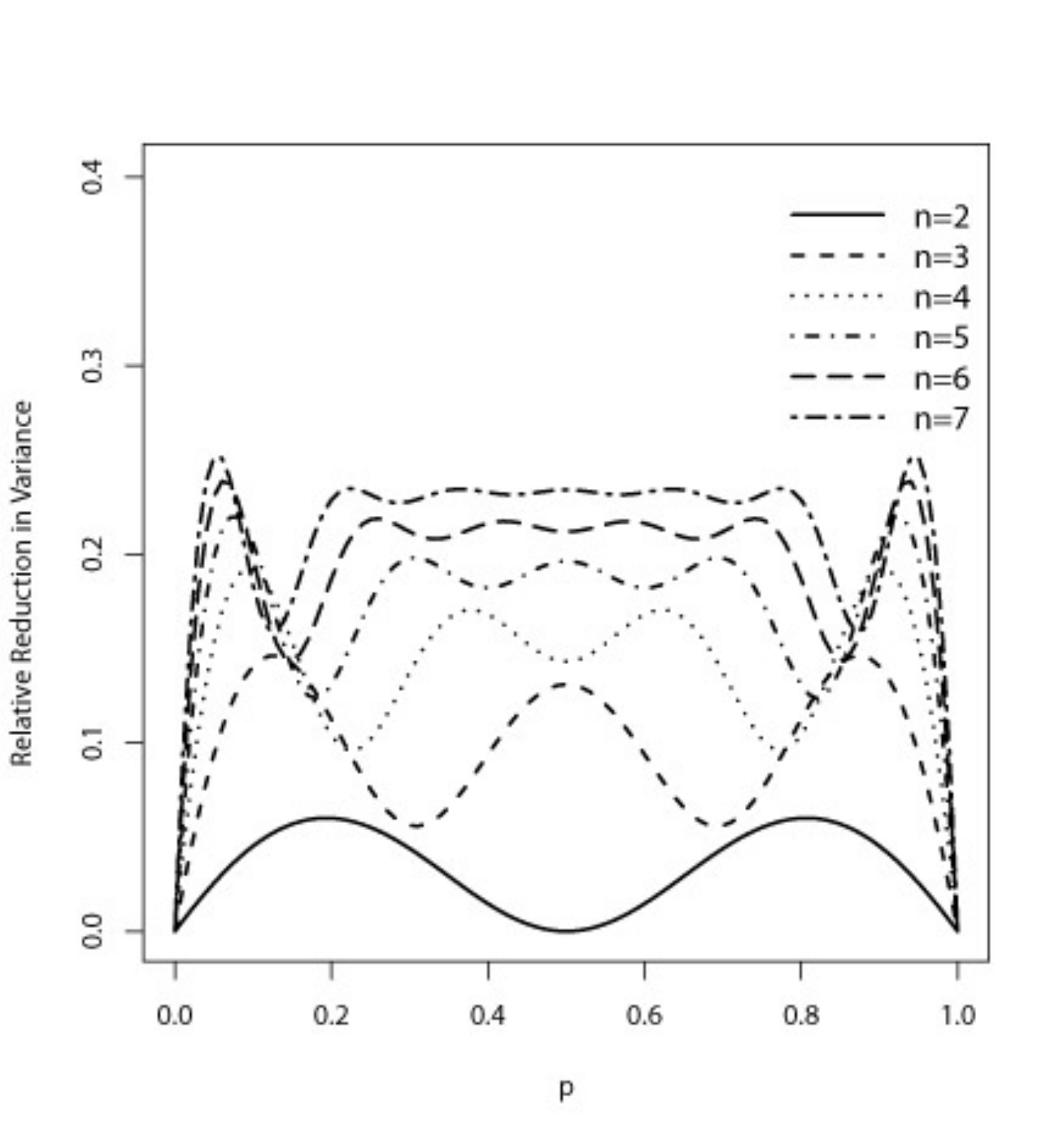} 
    (c) $\alpha_0=0.7$
  \end{minipage}
  \hfill
  \begin{minipage}[b]{0.45\linewidth}\centering
    \includegraphics[width=1\linewidth]{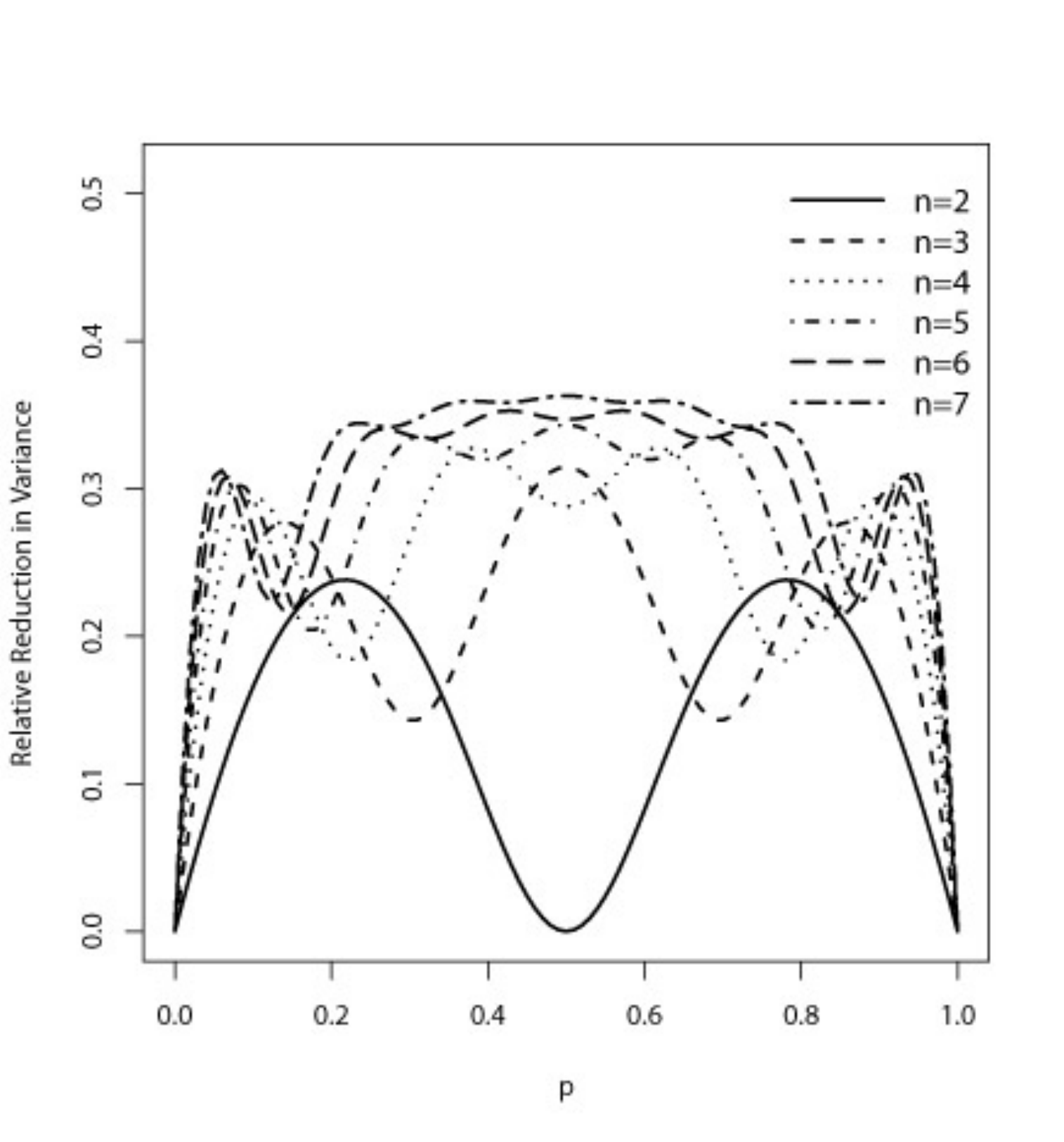} 
    (d) $\alpha_0=1$ 
  \end{minipage} 
  \caption{\large{$\mbox{RRV}(\hat{f}_{\PROS},\hat{f}_{\RSS})$ for
$n=2,\dots,7$ when $m=3$ and $\alpha_0=0,0.3,0.7,1$.}}
\label{fig:Figure2}
\end{figure}

\subsection{Comparing MISE's of $\hat{f}_{\PROS}$, $\hat{f}_{\RSS}$ and $\hat{f}_{\SRS}$}

In order to compare MISE$(\hat{f}_{\PROS})$ with
MISE$(\hat{f}_{\RSS})$ and MISE$(\hat{f}_{\SRS})$, following Chen
(1999), we consider  (a)  the standard Normal
distribution, (b)  the Gamma distribution with shape parameter 3
and scale parameter 1, and (c)  the standard Gumbel distribution. We
 use the Epanechnikov kernel in all estimates and the
bandwidth $h$ is  determined by
$$h={(4/3)}^{1/5}AN^{-1/5},$$
where $A=\min\{\mbox{standard devision of the sample,
interquartile range of the sample/1.34\}}$; see Silverman (1986).
For given $n$, $m$, $L$ and  different misplacement probabilities, we
use the following procedure to estimate the values of
MISE$(\hat{f}_{\PROS})$, MISE$(\hat{f}_{\RSS})$, and
MISE$(\hat{f}_{\SRS})$. For each estimator, the integrated square
error (ISE) $\int(\hat{f}(x)-f(x))^2dx$ is  calculated based on the 
corresponding SRS, imperfect RSS and  imperfect PROS  samples. Then,
the ISE of 5,000 PROS, RSS, and SRS estimates is  obtained. For
each procedure, the average of these 5,000 ISEs is   used as
an  estimate of the corresponding MISEs. The ratios
$$\mbox{RP=MISE}(\hat{f}_{\RSS})/\mbox{MISE}(\hat{f}_{\PROS})\quad  \text{and}\quad
\mbox{SP=MISE}(\hat{f}_{\SRS})/\mbox{MISE}(\hat{f}_{\PROS})$$  are 
obtained as the efficiency of $\hat{f}_{\PROS}$ with respect to
$\hat{f}_{\RSS}$ and $\hat{f}_{\SRS}$, respectively. Table \ref{table:Table2}  shows the values of RP and SP for these distributions with  different
values of $n$, $L$, $\alpha_0=0,0.3,0.5,0.7,1$, and $m=3$.

\tabcolsep=0.11cm
\begin{table}[h!]
\caption{{The efficiency of PROS density estimate with respect to
RSS (RP) and SRS (SP) for different values of $n$, $L$,
$\alpha_0$  and  Normal,
Gamma, and Gumbel distributions when $m=3$.}}\label{table:Table2}
\bigskip
\centering
\small
\begin{tabular}{l c c c c c c c}
\hline \hline
& & & \multicolumn{4}{c}{$\alpha_0$}\\
\cline{4-8}
& &  & 0 & 0.3 & 0.5 & 0.7 & 1\\
\cline{4-8}
Distributions & $n$ & $L$ & (RP , SP)& (RP , SP)& (RP , SP)& (RP , SP)& (RP , SP)\\
\hline
              & 6   & 4   & (1.012,1.030) & (1.000,0.993)  &  (1.022,1.084)  &  (1.100,1.281)  & (1.399,2.151) \\[-1ex]
              & 6   & 8   & (1.026,1.058) & (1.015,1.015)  &  (1.041,1.093)  &  (1.099,1.268)  & (1.309,1.960) \\[-1ex]
Normal        & 8   & 3   & (0.997,0.991) & (1.008,1.008)  &  (1.044,1.128)  &  (1.113,1.337)  & (1.408,2.453) \\[-1ex]
              & 8   & 6   & (0.998,1.014) & (1.006,1.018)  &  (1.042,1.119)  &  (1.106,1.310)  & (1.398,2.265) \\[.1cm]

\cline{2-8}
              & 6   & 4   & (1.021,1.022) & (1.004,1.009)  &  (1.010,1.069)  &  (1.060,1.190)  & (1.233,1.650) \\[-1ex]
              & 6   & 8   & (1.005,1.007) & (1.013,1.014)  &  (1.033,1.061)  &  (1.059,1.159)  & (1.160,1.486) \\[-1ex]
Gamma         & 8   & 3   & (1.018,1.054) & (0.992,1.010)  &  (1.025,1.113)  &  (1.077,1.278)  & (1.224,1.811) \\[-1ex]
              & 8   & 6   & (1.009,0.998) & (1.002,0.995)  &  (1.022,1.063)  &  (1.071,1.195)  & (1.173,1.546) \\[0.1cm]
\cline{2-8}
              & 6   & 4   & (0.985,1.040) & (0.984,1.023)  &  (1.013,1.095)  &  (1.102,1.284)  & (1.283,1.866) \\[-1ex]
              & 6   & 8   & (1.007,0.999) & (1.016,0.990)  &  (1.058,1.064)  &  (1.074,1.172)  & (1.239,1.619) \\[-1ex]
Gumbel        & 8   & 3   & (1.022,1.002) & (0.992,1.010)  &  (1.031,1.099)  &  (1.070,1.265)  & (1.269,1.936) \\[-1ex]
              & 8   & 6   & (0.996,0.984) & (1.009,1.007)  &  (1.037,1.075)  &  (1.092,1.208)  & (1.237,1.738) \\[0.1cm]
\hline

\end{tabular}
\end{table}

Form Table \ref{table:Table2}, it is seen  that as the misplacement probabilities
decrease the efficiency of PROS with respect to RSS and SRS
increases and, as we expect, the efficiency with respect to SRS is
more than RSS procedure. When the misplacement probabilities are
large, $\alpha_0< 0.5$, the three estimators have efficiency near
one.  The
efficiency of PROS with respect to RSS and SRS increases slightly
as $n$ increases (this increment is faster when $m=4$, results in which  are not presented here); however, the efficiency decreases as $L$
increases. The amount of efficiency for the Normal distribution
is higher than the Gamma and Gumbel distributions. For example,
when $n=8$, $L=3$, and $\alpha_0=1$ the efficiencies of PROS with
respect to SRS for the Normal, Gamma, and Gumbel distributions are
$145\%$, $81\%$, and $94\%$, respectively. We observe  that
the main parameter that controls the efficiency is the
misplacement probability matrix $\balpha$ or equivalently the ranking error.
When the ranking errors are high, there is no substantial
difference between $\hat{f}_{\PROS}$, $\hat{f}_{\RSS}$, and
$\hat{f}_{\SRS}$. However, as the ranking errors decrease, our
simulation results show that the PROS density estimate performs
better than  RSS and SRS density estimates in terms of MISE.

\subsection{Results under symmetry assumption}
To investigate the effect of estimating the symmetry point $\mu$
on the MISE of $\hat{f}^*_{\PROS}$, we consider four
distributions (a)  the standard Normal and (b)  Logistic distributions as
light tail distributions, (c)  t-student with 2 degrees of freedom, 
and (d)  the standard Laplace distributions as heavy tail
distributions. For each distribution, a perfect PROS sample of
size $N=nL$ with subset size $m$ are  generated and the four
symmetry point estimators given in (4) were calculated. Then,
MISE$(\hat{f}_{\PROS})$ and MISE$(\hat{f}^*_{\PROS})$ for
$\hat{\mu}_{i}$, $i=1,\dots,4$, were calculated and their ratios are  obtained as the efficiency of
$\hat{f}^*_{\PROS}(x,\hat{\mu}_{i})$'s with respect to
$\hat{f}_{\PROS}(x)$. The results for $m=3$, $n=6,8$, and $L=3,4$
are shown in Table \ref{table:Table4}. The last column  shows the
efficiency  of
$\hat{f}^*_{\PROS}$ with respect to $\hat{f}_{\PROS}$ when the
symmetry point is known.

We observe  that $\hat{f}^*_{\PROS}(x,\hat{\mu}_{3})$
performs the best  in all
cases  which  suggests using the
Hodges-Lehmann type estimator, $\hat{\mu}_{3}$, for estimating the
symmetry point.  For  Normal
distribution, the efficiencies of
$\hat{f}^*_{\PROS}(x,\hat{\mu}_{1})$ and
$\hat{f}^*_{\PROS}(x,\hat{\mu}_{3})$ with respect to
$\hat{f}_{\PROS}(x)$ are competitive. However, for $t(2)$ distribution,  which is  a heavy
tail distribution, it does not hold. On the other hand,
the efficiencies of $\hat{f}^*_{\PROS}(x,\hat{\mu}_{4})$ and
$\hat{f}^*_{\PROS}(x,\hat{\mu}_{3})$ with respect to
$\hat{f}_{\PROS}(x)$ are very close especially in heavy tail
distributions. Generally, we recommend using
$\hat{f}^*_{\PROS}(x,\hat{\mu}_{3})$ when it is assumed that the
underlying population distribution is symmetric. 
Comparing the efficiencies of $\hat{f}^*_{\PROS}(x,\mu=0)$ and $\hat{f}^*_{\PROS}(x,\hat{\mu}_{3})$
with respect to $\hat{f}_{\PROS}(x)$ indicates how much the
efficiency reduces when the symmetry point is estimated. This
reduction is larger when $n=6$ in comparison with $n=8$ (the
maximum value of reduction is about $11\%$ when $n=6$ and $L=3$ in
Normal distribution and the minimum value is $3\%$ when $n=8$ and
$L=4$ in Laplace distribution).

\begin{table}[H]
\caption{{The efficiencies of $\hat{f}^*_{\PROS}(x,\hat{\mu}_{i})$,
$i=1,\dots,4$ with respect to $\hat{f}_{\PROS}(x)$ and the efficiency of $\hat{f}^*_{\PROS}(x,\mu=0)$
with respect to $\hat{f}_{\PROS}(x)$ for $m=3$,
$n=6,8$, and $L=3,4$ when underlying distributions are the
standard Normal, Logistic, Laplace, and t-student with 2 degrees of
freedom.}}\label{table:Table4}
\bigskip
\centering
\begin{tabular}{l c c c c c c c}
\hline \hline
& & & \multicolumn{4}{c}{Estimators} & \\
\cline{4-7}
Distributions & $n$ & $L$ & $\hat{\mu}_{1}$ & $\hat{\mu}_{2}$ & $\hat{\mu}_{3}$ & $\hat{\mu}_{4}$ & $\mu=0$ \\
\hline
              & 6   & 3   & 1.176   & 1.062  &  1.212  &  1.155  & 1.364    \\[-1ex]
              & 6   & 4   & 1.185   & 1.057  &  1.225  &  1.165  & 1.353    \\[-1ex]
Normal        & 8   & 3   & 1.161   & 1.031  &  1.205  &  1.125  & 1.313    \\[-1ex]
              & 8   & 4   & 1.173   & 1.054  &  1.219  &  1.146  & 1.309    \\[0.1cm]
\cline{2-8}
              & 6   & 3   & 1.103   & 1.098  &  1.198  &  1.172  & 1.339    \\[-1ex]
              & 6   & 4   & 1.102   & 1.104  &  1.207  &  1.185  & 1.339    \\[-1ex]
Logistic      & 8   & 3   & 1.108   & 1.095  &  1.205  &  1.159  & 1.307    \\[-1ex]
              & 8   & 4   & 1.108   & 1.092  &  1.205  &  1.162  & 1.293    \\[0.1cm]
\cline{2-8}
              & 6   & 3   & 0.675   & 1.150  &  1.191  &  1.201  & 1.299    \\[-1ex]
              & 6   & 4   & 0.645   & 1.133  &  1.182  &  1.190  & 1.284    \\[-1ex]
t(2)          & 8   & 3   & 0.621   & 1.133  &  1.186  &  1.176  & 1.268    \\[-1ex]
              & 8   & 4   & 0.615   & 1.129  &  1.185  &  1.176  & 1.257    \\[0.1cm]
\cline{2-8}
              & 6   & 3   & 0.998   & 1.130  &  1.151  &  1.158  & 1.234    \\[-1ex]
              & 6   & 4   & 1.002   & 1.119  &  1.141  &  1.146  & 1.217    \\[-1ex]
Laplace       & 8   & 3   & 1.007   & 1.112  &  1.139  &  1.135  & 1.193    \\[-1ex]
              & 8   & 4   & 1.011   & 1.108  &  1.131  &  1.129  & 1.172    \\[0.1cm]
\hline

\end{tabular}
\end{table}

\section{Real Data Application}
In this section, we illustrate our method with a real data set collected by the Iranian Ministry of Jihade-Agricultural (IMJA) in 2005. Jafari Jozani et al.\ (2012) used this data set in a different context to examine the accuracy of several ratio estimators of the population mean based on  RSS design. The  data set contains  the information of the wheat yield and the total acreage of land which is planted in wheat for 304 cities in 31 provinces of Iran in 2005.  Wheat yield estimation is important for advanced planning and implementation of policies related to food distribution, import-export decision, etc. We  provide kernel density estimates of the distribution of  $Y$ = wheat yield (in ton) as the variable of interest by using  $X$ = total acreage of the planted land in wheat (in acre) as the auxiliary variable which can be used for  the ranking purpose. The correlation coefficient between $X$ and $Y$ is 0.786. For ease of  computations, we divided  the values of $Y$ by 100,000,000.


In order to estimate the density function of wheat yield, we regarded  this data set as a population and extracted PROS, RSS and SRS with replacement  samples of size $N=nL$ from the population. For each  design, the density estimates are obtained and the asymptotic variance estimates  are  calculated. Then,  this process is  repeated $M$ times and the average of density estimates at a fixed point are  considered  as the density estimates. In addition, for each design, the average of asymptotic variance estimates are also calculated for constructing asymptotic pointwise confidence bounds.
We  take  $n=3$, $m=4$, and $M=20$. The histogram of 304 records of $Y$ is shown in Figure \ref{fig:Figure5}. The PROS density estimate and its 95 percent asymptotic pointwise confidence bounds are shown in third column of  Figure \ref{fig:Figure5}. The SRS and RSS density estimates and their corresponding 95 percent  pointwise confidence bounds are also shown in   the 1st and 2nd columns   of Figure \ref{fig:Figure5}. In all cases, we used Epanechnikov kernel and the bandwidth was determined  as in Section 5.2. 
To obtain the probabilities of subsetting errors, we used the proposed algorithm in Section 4. We estimated the probabilities of subsetting errors for each 20 samples with SAE=0.001. The average of these estimates are given below 
$$\hat{\mathbf{\alpha}}=\left[ \begin{array}{ccc}
0.832 & 0.148 & 0.020 \\
0.148 & 0.737 & 0.116 \\
0.020 & 0.116 & 0.864 \end{array}
\right].$$
The standard deviations of the estimates vary between 0.05 and 0.19. Our estimates show that the probabilities of correct subsetting are much higher than the probabilities of incorrect subsetting and this is due to  the fact that $X$ and $Y$ are highly  correlated.
We observe  that the density estimates look  similar. However, the  confidence bounds for PROS design  are much narrower than the confidence bounds  obtained  by  SRS and RSS designs.  
\begin{figure}[H]
  \centering\vspace{-2cm}
  \includegraphics[width=18cm,height=12cm]{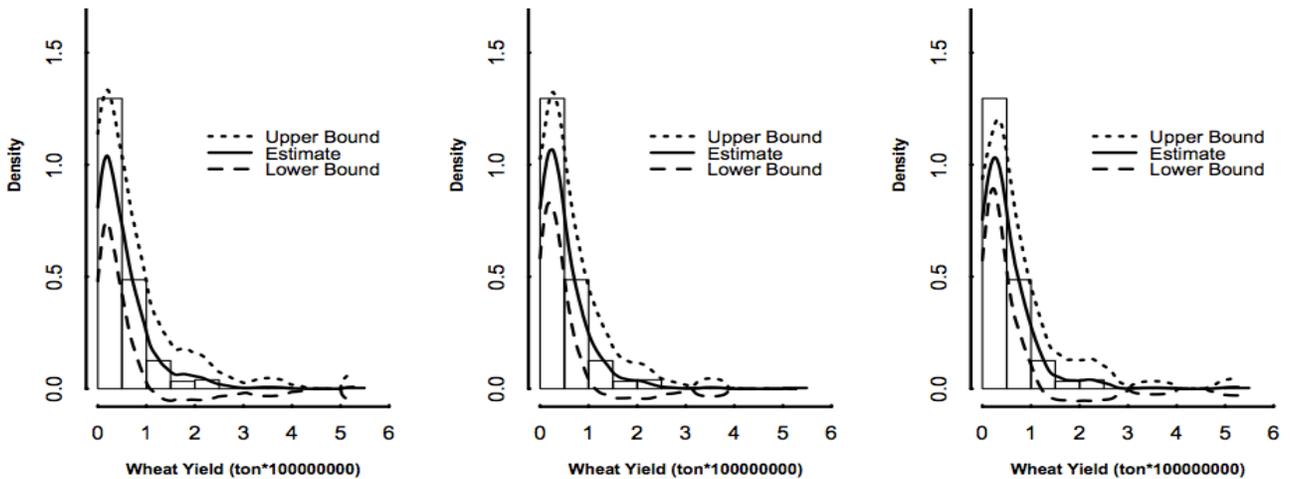}
  \vspace{-3.5cm}
  \caption{{The histogram of wheat yield (in ton $\times$ 100,000,000), the SRS, RSS and PROS kernel density estimates and their corresponding asymptotic 95 percent  pointwise confidence bounds. }}
  \label{fig:Figure5}
\end{figure}

%
%

\section*{Acknowledgments}
Mohammad Jafari Jozani  gratefully acknowledges the research support
of the Natural Sciences and Engineering Research Council of Canada.

\end{document}